\documentclass{amsart}
\usepackage{amsmath, amsthm, amssymb, mathrsfs}
\usepackage{url}
\usepackage{setspace}
\usepackage{url}
\usepackage{wasysym}
\usepackage{varioref}
\usepackage[ left=3cm, right=3cm]{geometry}

\allowdisplaybreaks

\renewcommand\theenumi{\roman{enumi}}

\theoremstyle{plain}
\newtheorem{thm}{Theorem}[section]
\newtheorem{corol}[thm]{Corollary}
\newtheorem{lemma}[thm]{Lemma}
\newtheorem{prop}[thm]{Proposition}

\newtheorem{remark}[thm]{Remark}

\theoremstyle{definition}
\newtheorem{defn}[thm]{Definition}

\newcommand{\n}{\|}
\newcommand{\len}{\left\|}
\newcommand{\pn}{\right\|}
\newcommand{\la}{\left\langle}
\newcommand{\ra}{\right\rangle}
\newcommand{\lee}{\left(}
\newcommand{\p}{\right)}
\newcommand{\vectwo}[2]{\left[\begin{array}{c} #1 \\ #2 \end{array}\right]}
\newcommand{\inv}[1]{\frac{1}{ #1 }}
\newcommand{\maxsym}{\vee}
\newcommand{\minsym}{\wedge}
\newcommand{\E}{\mathbb{E}}
\newcommand{\R}{\mathbb{R}}

\renewcommand{\P}{\mathbb{P}}

\newcommand{\F}{\mathcal{F}}
\newcommand{\eqtwo}[4]{\left\{\begin{array}{ll} #1 & #2  \\ #3 & #4 \end{array}\right.}

\newcommand{\Ep}[1]{\mathcal{E}_{#1}}
\newcommand{\Ca}[1]{\mathcal{#1}}
\newcommand{\Bo}[1]{\mathbb{#1}}
\newcommand{\norm}[1]{\left\|#1\right\|}

\newcommand{\vecfour}[4]{\left[\begin{array}{ll} #1 & #2  \\ #3 & #4 \end{array}\right]}

\begin{document}

\author{Mariusz G\'orajski}%\fnref{fn1}}

\title[Vector-valued stochastic delay equations]{Vector-valued stochastic delay equations - a weak solution and its Markovian representation}

\address{Mariusz G\'orajski, University of  \L \'od\'z, Rewolucji 1905 r. No. 41, 90-214 \L \'od\'z, Poland
\url{mariuszg@math.uni.lodz.pl}}

\begin{abstract}
 A class of stochastic delay equations in Banach space $E$ driven by cylindrical Wiener process is studied.
We investigate two concepts of solutions: weak and generalised strong, and give conditions under which they are equivalent. We present an evolution equation approach in a Banach space $\Ep{p}:=E\times L^p(-1,0;E)$ proving that the solutions can be reformulated as $\Ep{p}$-valued Markov processes. Based on the Markovian representation we prove the existence and continuity of the solutions. The results are applied to stochastic delay partial differential equations with an application to neutral networks and population dynamics. 
\end{abstract}


\maketitle
\section{Introduction} 
Let $H$ be a separable Hilbert space. For a Banach space $E$ and $p\geq 1$ define $\Ep{p}:=E\times L^ p(-1,0;E)$.  Let $W_H$ be an $H$-cylindrical Wiener process on a given probability space $(\Omega, (\F_t)_{t\geq 0},\F,\P)$. We shall consider stochastic delay equation in a Banach space $E$ of the form:
\begin{align}\label{SDE}
\left\{\begin{array}{l} 
dX(t)= \left(BX(t) + \Phi X_t +\phi(X(t),X_t)\right)dt +\psi(X(t),X_t)dW_H(t),\quad t>0;\\
X(0)=x_0;\\
X_0=f_0, 
\end{array}\right. \tag{SDE}
\end{align}
for initial conditions $[x_0, f_0]\in L^0((\Omega,\Ca{F}_0);\Ep{p})$,
where  $(X_t)_{t\geq 0}$ is a segment process formed from  $(X(t))_{t\geq 0}$ in the following way:
\begin{align*}
X_t(s):= X(t+s), \quad s\in [-1,0].
\end{align*}
Let us consider \eqref{SDE} with the following hypotheses:
\renewcommand\theenumi {\arabic{enumi}}
\renewcommand\theenumii {\alph{enumii}}
\renewcommand\labelenumi{(H\theenumi)}
\labelformat{enumi}{(H#1)}

\begin{enumerate}
\item \label{h1}  $B:D(A)\subset E\rightarrow E$ is a linear operator and generates a $C_0$-semigroup $(S(t))_{t\geq 0}$ on $E$,
 
\item \label{h2}  $\Phi$ 	is given by
Riemann-Stieltjes integral $$\Phi f=\int_{-1}^{0}d\eta f \textrm{ dla } f\in C([-1;0];E),$$ where $\eta:[-1,0]\to \Ca{L}(E)$ is of bounded variation,

\item \label{h3} $\phi:D(\phi)\subset \Ep{p}\to E$ is densely defined mapping and there exists $a\in L^{p}_{loc}(0,\infty)$ such that for all $t>0$ and $\mathcal{X}, \mathcal{Y} \in D(\phi)$,   
\begin{align*}
&\norm{S(t)\phi(\mathcal{X})}_E\leq a(t)(1+\|\mathcal{X}\|_{\Ep{p}}), \\
 &\norm{S(t)(\phi(\mathcal{X})-\phi(\mathcal{Y})}_E\leq a(t)\|\mathcal{X}-\mathcal{Y}\|_{\Ep{p}},
\end{align*} 

\item \label{h4} $\psi:D(\psi)\subset\Ep{p}\to \Ca{L}(H,E)$ is densely defined mapping such that for all $t>0$ and $\mathcal{X},  \mathcal{Y} \in D(\psi)$, $S(t)\psi(\mathcal{X})$ belongs to $\gamma(H,E)$ and there exists $b\in L^{p\maxsym 2}_{loc}(0,\infty)$ such that 
\begin{align*}
&\norm{S(t)\psi(\mathcal{X})}_{\gamma(H,E)}\leq b(t)(1+\|\mathcal{X}\|_{\Ep{p}}),\\
&\norm{S(t)(\psi(\mathcal{X})-\psi(\mathcal{Y})}_{\gamma(H,E)}\leq b(t)\|\mathcal{X}-\mathcal{Y}\|_{\Ep{p}},
\end{align*} 
\end{enumerate}
where $\gamma(H,E)$ is the space of $\gamma$-radonifying operators from $H$ to $E$ (see Section 2 in \cite{coxgorajski} or \cite{vanNeervenVeraarWeis}). \par

We use the evolution equation approach to the delay equation as given in the monograph of Batkai and Piazzera \cite{Batkai2005}. Thus, 
one can define a closed operator $\mathcal{A}$ on $\Ep{p}$ by 
\begin{align}\label{opA}
D(\mathcal{A})&=\{[x,f]'\in D(B)\times W^{1,p}(-1,0;E)\,:\, f(0)=x\}; \notag \\
\mathcal{A}&=\left[\begin{array}{cc} B & \Phi \\ 0 & \frac{d}{d\theta} \end{array}\right].
\end{align}
Under assumptions \ref{h1}-\ref{h2} $\mathcal{A}$ generates a $C_0$-semigroup $(\Ca{T}(t))_{t\geq 0}$  on $\mathcal{E}_p$ (see \cite{Batkai2005}, Theorem 3.29). Hence we can consider the following stochastic Cauchy problem corresponding to \eqref{SDE}:
\begin{align}\label{SDCP}
\left\{ \begin{array}{l} dY(t)=\mathcal{A}Y(t)dt +F(Y(t))dt+ G(Y(t))dW_H(t),\quad t\geq 0;\\
Y(0)=[x_0,f_0]', \end{array}\right. \tag{SDCP}
\end{align}
where 
\begin{align}\label{FG}
F(Y(t)):=[\phi(Y(t)),0]',\quad G(Y(t)):=[\psi(Y(t)),0]'.
\end{align}\par
Recall the classical result \cite{Chojnowska-Michalik1978}, where equivalence of solutions to the
stochastic delay equation and the corresponding abstract Cauchy problem has been
shown by Chojnowska-Michalik for $p = 2$
and $E$ finite-dimensional. For a general class of spaces including the $\Ep{p}$ spaces
the variation of constants formula for finite-dimensional delay equations with
additive noise and a bounded delay operator is discussed in Riedle \cite{Riedle2008}. For more references  see \cite{coxgorajski}.
We complement and extend result from \cite{coxgorajski} concerning existence and uniqueness of solution to \eqref{SDE} by adding non-linear part and introducing weak concept of solution to \eqref{SDE}. The line of thought we take is to prove existence and continuity of a weak solution to the stochastic Cauchy problem \eqref{SDCP} and then 
using the equivalence between weak solutions to \eqref{SDCP} and \eqref{SDE} we obtain corresponding results for the stochastic delay equation \eqref{SDE}.\par
A large class of stochastic partial differential equations (stochastic PDEs) with delay can be rewritten as  stochastic ordinary equations with delay \eqref{SDE} in infinite dimensional space $E$. Moreover, PDEs with delay are used in modelling phenomena inter alia in bioscience (see \cite{baker}, \cite{Bocharov2000183} and \cite{kot2001elements}) or in neural networks (see \cite{Lv20081590}). For some stochastic PDEs with delay e.g. stochastic delay reaction-diffusion equations with non-linearities given by the Niemycki operator (see Section 3.3 Examples) a generalized strong solution may not exist, whereas one can prove the existence and uniqueness of weak solution.

In the Da Prato and Zabczyk monograph \cite{DaPratoZabczyk} an extensive treatment of the stochastic Cauchy problem in Hilbert spaces is given. In the Banach space framework  stochastic Cauchy problem has been considered by Brze\'zniak \cite{Brzez95} and Van Neerven, Veraar and Weis \cite{vanNeervenVeraarWeis_SEEinUMD}. They both consider the case that $\mathcal{A}$ generates an analytic semigroup. 
\par
Following the semigroup approach let us consider the following variation of constants formula:
\begin{align}\label{voc}
Y(t) = \mathcal{T}(t)Y(0) + \int_{0}^{t}\mathcal{T}(t-s)F(Y(s))ds+\int_{0}^{t}\mathcal{T}(t-s)G(Y(s))dW_H(s),
\end{align}
where the precise definition of the stochastic integral above and the relevant theory on vector-valued stochastic integrals can be found in \cite{vanNeervenVeraarWeis_SEEinUMD}. A process satisfying \eqref{voc} is usually referred to as a \emph{mild solution}. The existence of a mild solution to \eqref{SDCP} is proved
by the Banach fixed-point theorem in Section \ref{s:SDE} (Theorem \ref{t:existSDCP}).

In Theorems \ref{t:Dvarcons} and \ref{t:Dvarcons2} we show that  a mild solution of \eqref{SDCP} is equivalent to a weak solution of \eqref{SDCP} and under some additional assumption they are equivalent to generalized strong solution of \eqref{SDCP}.  
Finally, using these theorems in Theorem \ref{t:rep} we state that weak solutions to \eqref{SDCP} and \eqref{SDE} are
equivalent. Combining all these results we obtain existence and continuity of weak solution of \eqref{SDE} (see Corollaries \ref{c:existSDE} and \ref{c:continuitySDE}).\par

The correspondence  between strong, weak and mild concept of solution to stochastic linear delay equations in Hilbert space has been considered by Liu using the properties of the Green operator in  
\cite{liu2008stochastic}.

The equivalence of solutions to \eqref{SDE} and to \eqref{SDCP} is useful
because the latter is a Markov process and can be studied in the framework of the stochastic abstract Cauchy problem; and one can answer questions concerning e.g.\
invariant measures of the solutions to \eqref{SDE} (see  \cite{chojnowska1995} and \cite{Bierkens2011} and reference therein), Feller property (see \cite{reissRiedleVanGaans_delaydiffeq}) and regularity of solutions (see Corollary \ref{c:continuitySDE} and \cite{vanNeervenVeraarWeis_SEEinUMD}). 

\section{The Stochastic Cauchy Problem}\label{s:SCP}
In the introduction we have mentioned that the stochastic delay equation \eqref{SDE} can be rewritten as a stochastic Cauchy problem. In this section we recall the result concerning different concept of solution to \eqref{SCP} form \cite{GorajskiSol}. Let $E$ be a Banach space  and $H$ be a separable Hilbert space, and let $A:D(A)\subset
E\rightarrow E$ be the generator of a $C_0$-semigroup $(T(t))_{t\geq0}$ on $E$.  The sun dual semigroup $(T^{\odot}(t))_{t\geq 0}$ defined as subspace semigroup by $T^{\odot}(t)=T^{*}(t)_{|E^{\odot}}$ defined on   $E^{\odot}=\overline{D(A^*)}$ is strongly continuous (see Section 2.6 in \cite{EngNag2000} and Chapter 1 in \cite{vanNeervenadjoint}). A generator $(A^{\odot},D(A^{\odot}))$ of the sun dual semigroup is given by $A^\odot=A^{*}_{|E^{\odot}}$ and $D(A^{\odot})=\{x^*\in D(A^*): A^*x^*\in E^{\odot}\}$.\par 
 Let $W_H$ be an $H$-cylindrical Brownian motion and following the monograph of Peszat and Zabczyk (see Section 9.2 and Remark 9.3 in \cite{PeszZab}) let
$F:D(F)\subset E\rightarrow E$ and $G:D(G)\subset E\rightarrow \Ca{L} (H,E)$ satisfy the following conditions: 

\renewcommand\theenumi {\Alph{enumi}}
\renewcommand\labelenumi{(H\theenumi)}
\labelformat{enumi}{(H#1)}

\begin{enumerate}
    \item \label{F} $D(F)$ is dense in $E$ and there exists $a\in L^1_{loc}(0,\infty)$ such that for all $t>0$ and $x,y\in D(F)$ we have  
\begin{align*}
&\|T(t)F(x)\|_E\leq a(t)(1+\|x\|_E),\\
&\|T(t)\left(F(x)-F(y)\right)\|_E\leq a(t)\|x-y\|_E,
\end{align*}
\item \label{G}  $D(G)$ is dense in $E$ and there exists $b\in L^2_{loc}(0,\infty)$ such that for all $t>0$ and $x,y\in D(G)$ we have  
\begin{align*}
&\|T(t)G(x)\|_{\gamma(H,E)}\leq b(t)(1+\|x\|_E),\\
&\|T(t)\left(G(x)-G(y)\right)\|_{\gamma(H,E)}\leq b(t)\|x-y\|_E.
\end{align*} 
\end{enumerate}
 Let us consider the following stochastic Cauchy problem in $E$:
\begin{align}\label{SCP}
\left\{ \begin{array}{rll} dY(t)&=AY(t)dt +F(Y(t))dt+ G(Y(t))dW_H(t), &t\geq 0;\\
Y(0)&=Y_0. \end{array}\right. \tag{SCP}
\end{align}

\renewcommand\theenumi {\roman{enumi}}
\renewcommand\labelenumi{(\theenumi) }

\begin{defn}\label{d:weaksolSCP}
An $H$-strongly measurable adapted process $Y$ is called a \emph{weak solution} to \eqref{SCP} if $Y$ is a.s. (almost surely)\ locally Bochner integrable and for
all $t> 0$ and all $x^*\in D(A^\odot)$:

\begin{enumerate}
\item $\la F(Y),x^*\ra$ is integrable on $[0,t]$  a.s.;
\item $G^*(Y)x^*$ is stochastically integrable on $[0,t]$;
\item  for almost all $\omega$ 
\begin{align*}
\langle  Y(t)-Y_0 ,x^*\rangle= \int_{0}^{t}\langle  Y(s),A^\odot x^* \rangle ds +\int_0^t\langle F(Y(s)), x^*\rangle ds +\int_{0}^{t} G^*(Y(s))x^* dW_H(s).
\end{align*}
\end{enumerate}
\end{defn}

In the next theorem we need stochastic integral for $\Ca{L}(H,E)$-valued process (for a definition and the following characterisation see \cite{vanNeervenVeraarWeis}). Recall that \textsc{umd} property stands for Unconditional Martigale Difference property and it says that all $L^p(\Omega;E)$-convergence, $E$-valued sequence of martingle difference are unconditionally convergent (see \cite{Garling_rmte} and \cite{vanNeervenVeraarWeis}). It turns out that for a Banach space with \textsc{umd} property we may characterise stochastic integrability in terms of $\gamma$-radonifying norm. More precisely a $H$-strongly measurable adapted process $\Psi:[0,t]\times\Omega\to \Ca{L}(H,E)$ is stochastically integrable with respect to cylindrical Wiener process $W_H$ if and only if  $\Psi$ represents $\gamma(L^2(0,t;H);E)$-valued random variable $R_{\Psi}$ given by 
\begin{align}\label{R}
\la R_{\Psi}f, x^*\ra=\int_0^t\la\Psi(s)f(s),x^*\ra ds \quad \textrm{a.s},
\end{align}
for every $f\in L^2(0,t;H)$ and for all $x^*\in E^*$. In this situation one has also the following Burkh\"older-Gundy-Davies type inequalities :
\begin{align}\label{BDG} \E\ \sup_{s\in[0,t]}\n\int_0^s\Psi(u)dW_H \n^p_E\eqsim_p\E\ \n R_\psi\n_{\gamma(L^2(0,t;H),E)}^p
\end{align}
for all $p>0$\footnote{For reals $A,B$ we use the notation $A\lesssim_p B$ to express the fact that there exists a constant $C>0$, depending on $p$, such that $A\leq CB$. We
write $A\eqsim_p B$ if $A\lesssim_p B\lesssim_p A$.}. To simplify terminology we say that process $\Psi$ is in $\gamma(L^2(0,t;H);E)$ a.s. iff $\Psi$ represents a random variable $R_\Phi$ given by \eqref{R}.\par
In \cite{Garling_rmte} Garling has characterised \textsc{umd} property in terms of two properties: \textsc{umd}$^-$ and \textsc{umd}$^+$. 
\begin{defn}\label{d:UMD-}
 A Banach space $E$ has \textsc{umd}$^-$ property, if for all $1<p<\infty$ there exists $\beta^-_p>0$ such that for all $E$-valued sequence of $L^p$-martingle difference $(d_n)_{n=1}^N$ and for all Rademacher sequence $(r_n)_{n=1}^N$ independent from $(d_n)_{n=1}^N$ we have the following inequality:
\begin{align}\tag{\textsc{umd}$^-$}\label{umd-}\E \ \norm{\sum_{n=1}^{N} d_n}_E^p\leq \beta_p^- \E \ \norm{\sum_{n=1}^{N} r_nd_n}_E^p. 
\end{align}
\end{defn}
 A Banach space $E$ has \textsc{umd}$^+$ property, if the reverse inequality to  \eqref{umd-} holds. Recall that class of \textsc{umd} Banach spaces is in class of reflexive spaces and includes  Hilbert spaces and $L^p$ spaces for $p\in(1,\infty)$. Moreover, class of \textsc{umd}$^-$ Banach spaces includes also non-reflexive $L^1$ spaces.\par
 To integrate processes with values in $L^1$ one needs a weakened notion of stochastic integral. In a Banach space $E$ with \textsc{umd}$^-$ property the following condition: $\Psi$ is in $\gamma(L^2(0,t;H),E)$ a.s. is sufficient for stochastic integrability of $\Psi$ (cf. \cite{vanNeervenVeraarWeis}  and \cite{vanNeervenVeraarWeis_SEEinUMD}).   
\begin{thm}[\cite{GorajskiSol}]\label{t:varcons}
Assume that $E$ has \textsc{umd}$^-$ property and conditions \ref{F} and \ref{G} are satisfied. Let $Y$ be an $E$-valued $H$-strongly measurable adapted process with almost all locally Bochner square integrable trajectories. If for all $t> 0$ the process:
\begin{align}\label{G1}
u\mapsto T(t-u)G(Y(u))
\end{align}
 is in $\gamma(L^2(0,t;H),E)$ a.s.
Then $Y$ is a weak solution to \eqref{SCP} if and only if $Y$ is a mild solution to \eqref{SCP} i.e. $Y$ satisfies, for all $t\geq 0$,
\begin{align}\label{voc2}
Y(t) &= T(t)Y_0 +\int_{0}^{t}T(t-s)F(Y(s))ds+\int_{0}^{t}T(t-s)G(Y(s))dW_H(s) \quad \textrm{a.s.}
\end{align}
\end{thm}

\begin{remark}\label{r:varcons}Fix $x^*\in D(A^*)$. Let $Y$ be a $E$-valued, strongly measurable adapted process with almost all locally Bochner square integrable trajectories. Then we have the following:
\begin{enumerate}
\item condition \ref{F} implies that  $E\ni x\mapsto \la F(x),x^*\ra \in \R{}$ is a Lipschitz function by Lemma 2.3 in \cite{GorajskiSol}. 
 
\item if \ref{G} holds then Lemma 2.3 in \cite{GorajskiSol} implies that the function $E\ni x\mapsto G^*(x)x^*\in H$ is Lipschitz-continuous. Hence the process  $G^*(Y)x^*$ is strongly measurable adapted with almost all locally square integrable trajectories. In particular, $G^*(Y)x^*$ is stochastically integrable on $[0,t]$ for all $t>0$.  
\item by \ref{F} and \ref{G} the functions $E\ni x\mapsto T(s)F(x)\in E$, $E\ni x\mapsto T(s)G(x)\in\gamma(H,E)$   are continuous. Hence processes $T(t-\cdot)F(Y(\cdot))$, $T(t-\cdot)G(Y(\cdot))$ are  adapted, strongly and $H$-strongly measurable, respectively, and $T(t-\cdot)F(Y(\cdot))$ has  trajectories locally Bochner square integrable a.s.  
\item since process in \eqref{G1} represent element from $\gamma(L^2(0,t;H),E)$ a.s. and $E$ has \textsc{umd}$^-$ property, stochastic integral in \eqref{voc2} is well defined.
\end{enumerate} 
\end{remark}
 A generalised strong solution to \eqref{SCP} is defined and its equivalence to a mild solution of \eqref{SCP} is proven in \cite{coxgorajski}. 
\begin{defn}\label{d:strongsolSCP}
A strongly measurable adapted process $Y$ is called a \emph{generalized strong solution} to \eqref{SCP} if $Y$ is, almost surely, locally Bochner integrable and for
all $t> 0$:
\begin{enumerate}
\item $\int_{0}^{t} Y(s)ds \in D(A)$ a.s.,
\item $F(Y)$ is Bochner integrable in $[0,t]$ a.s.,
\item $G(Y)$ is stochastically integrable on $[0,t]$,
\end{enumerate}
and
$$Y(t) - Y_0 = A \int_{0}^{t} Y(s)ds + \int_{0}^{t} F(Y(s)) ds+\int_{0}^{t} G(Y(s)) dW_H(s)\quad a.s.$$
\end{defn}
The equivalence of mild, weak and generalised strong solution to \eqref{SCP} has been established in \cite{GorajskiSol}. First we recall the hypotheses 
\renewcommand\theenumi {\Alph{enumi}}
\renewcommand\theenumii {\roman{enumii}}
\renewcommand\labelenumi{(H\theenumi') }
\labelformat{enumi}{(H#1')}
\begin{enumerate}
    \item \label{F'} Assume that \ref{F} is satisfied and for all $t>0$ and all $g\in L^1(0,t;E)$ the function  $F(g)$ is Bochner integrable on $[0,t]$. 
\end{enumerate}

If $F$ is a Lipschitz function then \ref{F'} is satisfied.
\renewcommand\theenumi {\roman{enumi}}
\renewcommand\labelenumi{(\theenumi)}

\begin{thm}[\cite{GorajskiSol}]\label{t:varcons2}
Assume that $E$ has \textsc{umd}$^-$ property and conditions \ref{F'} and \ref{G} are satisfied. Let $Y$ be an $E$-valued $H$-strongly measurable adapted process with locally Bochner square integrable trajectories a.s. If for all $t> 0$ the processes:
\begin{align}\label{G2}
u\mapsto G(Y(u)),\quad u\mapsto T(t-u)G(Y(u)), \quad u\mapsto \int_0^{t-u}T(s)G(Y(u,\omega))ds
\end{align}
 are in $\gamma(0,t;H,E)$ a.s.
Then the following condition are equivalent:
\begin{enumerate}
	\item $Y$ is a generalised strong solution of \eqref{SCP}.
    \item $Y$ is a weak solution of \eqref{SCP}.
	 	 \item $Y$ is a mild solution of \eqref{SCP}.
\end{enumerate}  
\end{thm}

\section{The Stochastic Delay Equation}\label{s:SDE}

\subsection{The variation of constants formula}
We now turn to the stochastic delay equation \eqref{SDE} as presented in the introduction and to the related stochastic Cauchy problem \eqref{SDCP} on page \pageref{SDCP}. 

We assume \ref{h1}-\ref{h2}. Then, the operator $\mathcal{A}$ (cf. \eqref{opA}) generates the strongly continuous semigroup $(\mathcal{T}(t))_{t\geq 0}$ on $\Ep{p}$ with a $L^p$-norm given by $\len [x,f]\pn_{\Ep{p}}=\lee\len x\pn^p_E+\len f\pn^p_{L^p(-1,0;E)}\p^\inv{p}$(see Theorem 3.29 in \cite{Batkai2005}). We shall define the projections $\pi_1:\Ep{p}\rightarrow E$ and $\pi_2 : \Ep{p} \rightarrow L^p(-1,0;E) $ as follows: $\pi_1 \left[x,f\right]'=x$; $\pi_2 \left[x,f\right]'=f$.

The following property of $(\mathcal{T}(t))_{t\geq 0}$ is intuitively obvious and useful in the following:
\begin{align}\label{propT}
\left(\pi_2 \mathcal{T}(t)\left[\begin{array}{c} x \\ f \end{array}\right]\right)(u) &= \pi_1 \mathcal{T}(t+u)\left[\begin{array}{c} x \\ f \end{array}\right],
\end{align}
for $[x,f]'\in \Ep{p}, u\in[-1,0], t>-u$ (for a proof see \cite{Batkai2005}, Proposition 3.11). \par

\begin{lemma}\label{l:eqSolSCPD}
Assume that \ref{h1} and \ref{h3}-\ref{h4} hold.
Then for $F:\Ep{p}\to \Ep{p}$ and $G:\Ep{p}\to\Ca{L}(H,\Ep{p})$ given by \eqref{FG} there exist 
$\tilde{a}\in L_{loc}^p(0,\infty)$, $\tilde{b}\in L_{loc}^{2\maxsym p}(0,\infty)$ such that for all $t>0$ we have
\begin{enumerate}
\item if \ref{h2} holds, then
\begin{align}\label{r:DFa}
&\norm{\pi_1\Ca{T}(t)F(\Ca{X})}_{E}\leq \tilde{a}(t)(1+\norm{\Ca{X}}_{\Ep{p}}), \\
&\norm{\pi_1\Ca{T}(t)(F(\Ca{X})-F(\Ca{Y}))}_{E}\leq \tilde{a}(t)\norm{\Ca{X}-\Ca{Y}}_{\Ep{p}} \label{r:DF}
\end{align}
for all $\Ca{X},\Ca{Y}\in D(\phi)$, and 
\begin{align} 
\label{r:DGa}
&\norm{\pi_1\Ca{T}(t)G(\Ca{X})}_{\Ca{L}(H,E)}\leq \tilde{b}(t)(1+\norm{\Ca{X}}_{\Ep{p}}), \\
&\norm{\pi_1\Ca{T}(t)(G(\Ca{X})-G(\Ca{Y}))}_{\Ca{L}(H,E)}\leq \tilde{b}(t)\norm{\Ca{X}-\Ca{Y}}_{\Ep{p}},
\label{r:DG}
 \end{align}
for all $\Ca{X},\Ca{Y}\in D(\psi)$.
\item if \ref{h2} holds, $E$ is a Hilbert space, then one can replace the $\Ca{L}(H,E)$-norm  in \eqref{r:DGa} and \eqref{r:DG} with the $\gamma(H,E)$-norm \footnote{If $H$ and $E$ are Hilbert spaces then we have $\gamma(H,E)=\Ca{L}_2(H,E)$, where $\Ca{L}_2(H,E)$ is a space of Hilbert-Schmidt operators.}. 
\item if the delay operator is bounded i.e. $\Phi\in\Ca{L}(L^p(-1,0;E),E)$, then one can replace the $\Ca{L}(H,E)$-norm  in \eqref{r:DGa} and \eqref{r:DG} with the $\gamma(H,E)$-norm.
\end{enumerate}
\end{lemma}

\begin{proof}
We only prove \eqref{r:DG}. The same proofs works for  \eqref{r:DFa}-\eqref{r:DGa}.\\
\textbf{(i).}  
The following formula defines a semigroup on $\Ep{p}$:  $\Ca{T}_0(t)=\vecfour{S(t)}{0}{\Ca{S}_t}{T_l(t)}$ for every $t\geq 0$, where  $\lee T_l(t)\p _{t\geq 0}$ is the left translation semigroup on $L^p(-1,0;E)$ and $\Ca{S}_s\in\Ca{L}(E,L^p(-1,0;E))$ is given by $$(\Ca{S}_sx)(\theta)=\eqtwo{0}{\theta\in(-1,-s\maxsym -1)}{S(\theta+s)x}{\theta\in(-s\maxsym-1,0)}.$$
for all $s\geq0$ and all $x\in E$ (cf. Theorem 3.25 in \cite{Batkai2005}). Let $\lee\Ca{A}_0=\vecfour{B}{0}{0}{\frac{d}{d\theta}},D(\Ca{A}_0)=D(\Ca{A})\p$ be the generator of $(\Ca{T}_0(t))_{t\geq0}$. Recall that the delay semigroup $\lee \Ca{T}(t)\p_{t\geq 0}$ can be built by the Miyadera-Voight perturbation theorem as a semigroup generated by  additive perturbation of generator $\Ca{A}_0$ of the form: $\Ca{A}=\Ca{A}_0+\vecfour{0}{\Phi}{0}{0}$ (cf. Theorem 1.37 in \cite{Batkai2005}). Moreover, we have   the variation of constant formula:
\begin{align}\label{vocf}
\Ca{T}(t)\Ca{X}=\Ca{T}_0(t)\Ca{X}+\int_0^t\Ca{T}(t-s)\vecfour{0}{\Phi}{0}{0}\Ca{T}_0(s)\Ca{X}ds, \quad \Ca{X}\in D(\Ca{A}).
\end{align}
Then, for all $t>0$ and for every $\Ca{X}=[x,g]^{'}\in D(\Ca{A})$ we have:
\begin{align}\label{r:piFUS_M-Voigta}
\n\pi_1\Ca{T}(t)\Ca{X}\n_{E} &\leq\len S(t)x\pn_{E}+\n\int_0^t\pi_1\Ca{T}(t-s)\vectwo{\Phi(\Ca{S}_sx+T_l(s)g)}{0}ds\n_{E}.
\end{align}
Since $\Phi$ is given by Riemann-Stieltjes integral (cf. \ref{h2}), one can apply the Fubini theorem and the H\"older inequality to 
\begin{align}\label{r:h82}
&\int_0^t\norm{ \Phi (\Ca{S}_sx+T_l(s)g)}_Eds=\int_0^{1\minsym t}\norm{ \Phi (\Ca{S}_sx+T_l(s)g)}_Eds+\int_1^{t\maxsym 1}\norm{ \Phi (\Ca{S}_sx+T_l(s)g)}_Eds\\&\hspace{1.5cm}\leq
 |\eta|(-1,0)\lee\norm{g}_{L^p(-1,0;E)}+\int_0^{1\minsym t}\norm{ S(s)x}_Eds\p+\int_{-1}^0\int_{1}^{t\maxsym 1}\norm{ S(s+\theta)x}_E dsd|\eta|(\theta)\notag\\
&\hspace{1.5cm}\leq
 |\eta|(-1,0)\lee\norm{ g}_{L^p(-1,0;E)}+\int_0^t\norm{ S(s)x}_Eds\p.\notag
\end{align}
Thus
\begin{align*}
&\norm{\pi_1\int_0^t\Ca{T}(t-s)\vectwo{\Phi(\Ca{S}_sx+T_l(s)g)}{0}ds}_{E}\leq M_{\Ca{T}}(t)|\eta|(-1,0)\lee\norm{ g}_{L^p(-1,0;E)}+\int_0^t\norm{ S(s)x}_Eds\p, \notag 
\end{align*}
where $M_{\Ca{T}}(t)=\sup_{s\in[0,t]}\n\Ca{T}(u)\n_{\Ca{L}(\Ep{p})}$.
On substituting the above estimation into \eqref{r:piFUS_M-Voigta} we obtain
\begin{align}\label{r:pi1}
\norm{\pi_1\Ca{T}(t)\Ca{X}}_{\Ep{p}} \leq&\norm{S(t)x}_E+M_{\Ca{T}}(t)|\eta|(-1,0)\lee\n g\n_{L^p(-1,0;E)}+\int_0^t\n S(s)x\n_Eds\p,
\end{align}
for all $\Ca{X}=[x,g]^{'}\in D(\Ca{A})$. Since  $D(\Ca{A})$ is dense in~$\Ep{p}$  and $\Ca{T}(t)$ is bounded, \eqref{r:pi1} holds for all $\Ca{X}\in\Ep{p}$. In particular, by second inequality in  \ref{h4} and the inequality $\n\cdot\n_{\Ca{L}(H,E)}\leq\n\cdot\n_{\gamma(H,E)}$ one has that, for all $\Ca{X},\Ca{Y}\in \Ep{p}$,   
\begin{align}\notag
\norm{\pi_1\Ca{T}(t)(G(\Ca{X})-G(\Ca{Y}))}&_{\Ca{L}(H,E)} \leq\norm{S(t)\lee\psi(\Ca{X})-\psi(\Ca{Y})\p}_{\Ca{L}(H,E)}\\&\quad+M_{\Ca{T}}(t)|\eta|(-1,0)\int_0^t\norm{ S(s)\lee\psi(\Ca{X})-\psi(\Ca{Y})\p}_{\Ca{L}(H,E)}ds\label{r:o}\\ \notag
&\leq\lee b(t)+M_{\Ca{T}}(u)|\eta|(-1,0)\int_0^tb(s)ds\p\n\Ca{X}-\Ca{Y}\n_{\Ep{p}}.
\end{align}
Let us notice that the function $\tilde{b}$ defined as $\tilde{b}(t):=b(t)+M_{\Ca{T}}(t)|\eta|(-1,0)\int_0^tb(s)ds$ for a.e. (almost every) $t\geq 0$ belongs to $L^{2}_{loc}(0,\infty)$.
The proof of \eqref{r:DFa}-\eqref{r:DF} follows between the same lines with $\tilde{a}(t)=a(t)+M_{\Ca{T}}(t)|\eta|(-1,0)\int_0^ta(s)ds$ for a.e. $t\geq 0$.\par

\textbf{(ii).} Assume now that $E$ is a Hilbert space. Let $(h_n)_{n\geq 1}$ be an orthonormal system in~$H$. By \eqref{r:pi1} and \ref{h4} we obtain, for all $\Ca{X},\Ca{Y}\in \Ep{p}$,
\begin{align*}
\n\pi_1\Ca{T}(t)(G(\Ca{X})-G(\Ca{Y}))&\n_{\Ca{L}_2(H,E)}=\lee\sum_{n=1}^\infty\norm{\pi_1\Ca{T}(t)(G(\Ca{X})-G(\Ca{Y})h_n}^2_{E}\p^\inv{p}\\&\leq 
\lee\sum_{n=1}^\infty \norm{S(t)(\psi(\Ca{X})-\psi(\Ca{Y}))h_n}^2_{E}\p^\inv{2}\\&\quad+M_{\Ca{T}}(t)|\eta|(-1,0)\lee\sum_{n=1}^\infty\lee\int_0^t\norm{ S(s)(\psi(\Ca{X})-\psi(\Ca{Y}))h_n}_{E}ds\p^2\p^\inv{2}\\
&\leq \len S(t)(\psi(\Ca{X})-\psi(\Ca{Y}))\pn_{\Ca{L}_2(H,E)}\\&+
M_{\Ca{T}}(t)|\eta|(-1,0)\int_0^t\norm{ S(s)(\psi(\Ca{X})-\psi(\Ca{Y}))}_{\Ca{L}_2(H,E)}ds \\
&\leq \lee b(t)+M_{\Ca{T}}(t)|\eta|(-1,0)\int_0^t b(s)ds \p \norm{\Ca{X}-\Ca{Y}}_{\Ep{2}},
\end{align*}
where in the second inequality we use the Minkowski integral inequality.\par

\textbf{(iii).} Assume that $\Phi\in\Ca{L}(L^p(-1,0;E),E)$. First, let us observe that if $\Phi\in\Ca{L}(L^p(-1,0;E),E)$ and \ref{h1} hold, then  $\Ca{A}$, defined by \eqref{opA} generates strongly continuous semigroup on $\Ep{p}$ (cf. \cite{Batkai2005} ).  For any  $\Ca{X,Y}\in \Ep{p}$  let $\Lambda=\psi(\Ca{X})-\psi(\Ca{Y})$ denote an operator from $\Ca{L}(H,E)$.  Then using \eqref{vocf} we obtain
that for all $t>0$ \begin{align}\label{1ii}
\norm{\pi_1\Ca{T}(t)(G(\Ca{X})-G(\Ca{Y}))}_{\gamma(H,E)}\leq \norm{ S(t)\Lambda}_{\gamma(H,E)}+\norm{\pi_1\int_0^t\Ca{T}(t-s)[\Phi\Ca{S}_s\Lambda,0]'ds}_{\gamma(H,E)}. 
\end{align} 
By boundedness of $\Phi$ and due to the ideal property of the $\gamma$-radonifying operators and therefore by $\gamma$-Fubini isomorphism (cf. Proposition 2.6 in \cite{vanNeervenVeraarWeis}) between the spaces: $L^p(-1,0;\gamma(H,E))$ and $\gamma(H,L^p(-1,0;E))$  
 we can estimate the second term on right hand side of \eqref{1ii} as follows
\begin{align}\label{3ii}
&\len\pi_1\int_0^t\Ca{T}(t-s)[\Phi\Ca{S}_s\Lambda,0]'ds\pn_{\gamma(H,E)}\leq
\pi_1\int_0^t\norm{\Ca{T}(t-s)[\Phi\Ca{S}_s\Lambda,0]'}_{\gamma(H,E)}ds\\ &\leq
\int_0^t\norm{\Ca{T}(t-s)}_{\Ca{L}(\Ep{p})}\norm{\Phi}_{\Ca{L}(L^p(-1,0;E),E)}\norm{\Ca{S}_s\Lambda}_{\gamma(H,L^p(-1,0;E))}ds\notag\\ &\leq M_{\Ca{T}}(t)\norm{\Phi}_{\Ca{L}(L^p(-1,0;E),E)}\lee\int_{-s\maxsym-1}^0\norm{ S(s+\theta)\Lambda}_{\gamma(H,E)}^p d\theta\p^\inv{p}
\len\Ca{X}-\Ca{Y}\pn_{\Ep{p}}.\notag
\\ &\notag\leq M_{\Ca{T}}(t)\norm{\Phi}_{\Ca{L}(L^p(-1,0;E),E)}\int_0^t\lee\int_{0}^sb^p(r)dr\p^\inv{p}ds\n\Ca{X}-\Ca{Y}\n_{\Ep{p}},
\end{align} 
where in the last inequality we use assumption \ref{h4}. 
Finally, using  \eqref{1ii}-\eqref{3ii} we obtain
\begin{align*}
\norm{\Ca{T}(t)(G(\Ca{X})-G(\Ca{Y}))}_{\gamma(H,\Ep{p})}&\leq \Bigg( b(t)+\lee\int_0^tb^p(s)ds\p^\inv{p}\\&\qquad+M_{\Ca{T}}(t)|\eta|(-1,0)t\lee\int_0^tb^p(s)ds\p^\inv{p}\Bigg)\norm{\Ca{X}-\Ca{Y}}_{\Ep{p}}.
\end{align*} 
 for all $t> 0$.
\end{proof}

Now we shall consider \eqref{SDCP} in $\Ep{p}=E \times L^p(-1,0;E)$ where $E$ is a type 2 Banach space  with \textsc{umd}$^-$ property. Recall that a Banach space $E$ is said to have {\em type $p\in [1,2]$} if there
exists a constant $C\ge 0$ such that for all
finite choices $x_1,\dots, x_k\in F$ we have
$$ \Big(\E\ \Big\n \sum_{j=1}^k \gamma_j x_j \Big\n_E^2\Big)^\frac12 \le C
\Big(\sum_{j=1}^k \n x_j\n_E^p\Big)^\frac1p,$$
where $(\gamma_j)_{j\geq 1}$ is a sequence of independent standard Gaussians. We note that Hilbert spaces have type $2$ and $L^p$-spaces with $p\in [1,\infty)$ have type $\min\{p,2\}$. For more details we refer the reader to \cite{Albiac2006}. In the next theorem we will need the following embedding:
\begin{align}\label{emb}
L^2(0,t;\gamma(H,E))\hookrightarrow \gamma(L^2(0,t;H),E),
\end{align}
which holds for a type 2 Banach space $E$ (see p.\ 1460 in \cite{vanNeervenVeraarWeis}).
The first result concerning stochastic Cauchy problem for delay equation \eqref{SDCP} says that its weak solutions and  mild solutions are equivalent.
\begin{thm}\label{t:Dvarcons}
Let $E$ be a type $2$ \textsc{umd}$^-$ Banach space and let $p\in [1,\infty)$. Assume that \ref{h1} and \ref{h3}-\ref{h4} hold and that one of the following is satisfied:
\begin{enumerate}
\item[(a)]$\Phi\in\Ca{L}(L^p(-1,0;E),E)$;
\item[(b)] \ref{h2} holds and either $H$ has a finite dimension or $E$ is a Hilbert space.
\end{enumerate}
 Let us  consider \eqref{SDCP}; i.e.\ let $\mathcal{A}$ defined by \eqref{opA} be the generator of the
$C_0$-semigroup $(\mathcal{T}(t))_{t\geq 0}$ on $\mathcal{E}_p=E\times L^p(-1,0;E)$. Let $F:\Ep{p}\to \Ep{p}$ and $G:\Ep{p}\to\Ca{L}(H,\Ep{p})$ be given by \eqref{FG}. 
Let $Y:[0,\infty)\times \Omega \rightarrow \mathcal{E}_p$ be a strongly measurable, adapted process satisfying \footnote{From now on we  denote by $\Bo{SL}_{\Ca{F}}^q(0,t;\Ep{p})$ for some $t>0$ and $q\geq 1$ a Banach space of strongly measurable, adapted process $Y$ with the norm 
$\n Y \n_{\Bo{SL}_{\Ca{F}}^q(0,t;\Ep{p})}=\sup_{s\in[0,t]}\lee\E\ \n Y(s)\n_{\Ep{p}}^{q}\p^{q}$.}
\begin{align}\label{intY}
\sup_{s\in[0,t]}\E\ \n Y(s)\n_{\mathcal{E}_p}^{2} < \infty \quad \textrm{ for all } t>0.
\end{align}
Then  $Y$ is a weak solution to \eqref{SDCP} if and only if $Y$ is a solution to: 
\begin{align}\label{Dvarcon}
Y(t) &= \mathcal{T}(t)\vectwo{x_0}{f_0} + \int_{0}^{t} \mathcal{T}(t-s)F(Y(s))ds+\int_{0}^{t} \mathcal{T}(t-s)G(Y(s))dW_H(s),
\end{align}
a.s. for all $t\geq 0$.
\end{thm}
\begin{proof}
Let $Y:[0,\infty)\times \Omega \rightarrow \mathcal{E}_p$ be a strongly measurable, adapted process satisfying \eqref{intY}.
We apply Theorem \ref{t:varcons} to obtain the above assertion.  Thus we need to check conditions \ref{F} and \ref{G} with $F$ and $G$ defined by \eqref{FG}, and that the processes given by \eqref{G1} in that theorem is an element of $\gamma(L^2(0,t;H),\Ep{p})$ a.s.\ for all $t>0$. Let $t>0$ be fixed.\par
First let us notice that  if $\Phi\in\Ca{L}(L^p(-1,0;E),E)$ or \ref{h2} holds, $E$ is a Hilbert space, then by Lemma \ref{l:eqSolSCPD}.(iii) or Lemma \ref{l:eqSolSCPD}.(ii), respectively,  conditions  \ref{F} i \ref{G} from Theorem \ref{t:varcons} hold.
If $H$ is a finite dimensional space, then $\gamma(H,E)$ is isomorphic with $\Ca{L}(H,E)$. Hence by Lemma \ref{l:eqSolSCPD}.(i) it follows that conditions  \ref{F} i \ref{G} from Theorem \ref{t:varcons} are satisfied.   Since $Y\in\Bo{SL}_{\Ca{F}}^2(0,t;\Ep{p})$ for all $t>0$, in particular $Y$ has, almost surely, trajectories square integrable.
It is enough to check condition \eqref{G1} in Theorem \ref{t:varcons}. In the proof we use the following inequality:  
\begin{align}\label{gammanorm} \norm{\Ca{G}}_{\gamma(L^2(0,t;H),\Ep{p})}\leq C\lee\norm{\pi_1\Ca{G}}_{L^2(0,t;\gamma(H,E))}+\norm{\pi_2\Ca{G}}_{L^p(-1,0; L^2(0,t;\gamma(H,E)))}\p,
\end{align}
which holds by embedding \eqref{emb} and by $\gamma$-Fubini isomorphism between $\gamma(L^2(0,t,H),L^p(-1,0;E))$ and $L^p(-1,0; L^2(0,t;\gamma(H,E)))$ (cf. Proposition 2.6 in \cite{vanNeervenVeraarWeis}).\par
By Lemma \ref{l:eqSolSCPD} we obtain
\begin{align}\notag
\lee\E\ \norm{ u\mapsto \pi_1 \Ca{T}(t-u)G(Y(u)) }^2_{L^2(0,t;\gamma(H,E))}\p^\inv{2}&\leq\lee \E\ \int_0^t \tilde{b}^2(t-u)(1+\norm{ Y(u)}_{\Ep{p}})^2du \p^\inv{2} \\ \label{pi1a}
&\leq\norm{ \tilde{b}}_{L^2(0,t)}(1+\norm{ Y}_{\Bo{SL}_{\Ca{F}}^2(0,t;\Ep{p})})<\infty.
\end{align}
Next, by  \eqref{propT} and by Lemma \ref{l:eqSolSCPD} we see that for almost all $\omega$ 
\begin{align}\label{pi2a}
&\norm{(\theta,u)\mapsto(\pi_2 \Ca{T}(t-u)G(Y(u)))(\theta)}_{L^p(-1,0;L^2(0,t;\gamma(H,E)))}\\ \notag
&\hspace{3cm}=\lee\int_{-1}^{0} \lee\int_0^{t+\theta}\norm{ \pi_1 \Ca{T}(t-u+\theta) G(Y(u))}_{\gamma(H,E)}^2du\p^{\frac{p}{2}}d\theta\p^\inv{p}\\ \notag
&\hspace{3cm}\leq \lee\int_{-1}^0\lee\int_{0}^{t+\theta} \tilde{b}^2(t-u+\theta)\lee 1+\norm{Y(u)}_{\Ep{p}}\p^2du\p^{\frac{p}{2}}d\theta\p^\inv{p}\\
&\hspace{3cm}\leq \lee\int_0^t\lee\int_{-1}^0g^{p\maxsym 2}(t-u+\theta)d\theta\p^{\frac{2}{p\maxsym 2}}\lee 1+\norm{ Y(u)}_{\Ep{p}}\p^2du\p^\inv{2}, \notag 
\end{align}
where $g(u)=1_{\{u\geq 0\}}\tilde{b}(u), u\in\R{}$. In the last inequality in \eqref{pi2a} if $p\leq 2$  we apply Jensen's inequality for integral with respect to $\theta$ and then Fubini's theorem, and if $p> 2$ we use Minkowski's integral inequality.  
Moreover, notice that for a.e. $u\in[0,t]$  we obtain
$$\lee\int_{-1}^0g^{p\maxsym 2}(t-u+\theta)d\theta\p^{\frac{2}{p\maxsym 2}}\leq\norm{\tilde{b}}^2_{L^{p\maxsym 2}(0,t)}.$$
We conclude from the above inequality and  \eqref{pi2a} that 
\begin{align}\label{pi2a_2}&\norm{ (\theta,u)\mapsto(\pi_2 \Ca{T}(t-u)G(Y(u)))(\theta) }_{L^p(-1,0; L^2(0,t;\gamma(H,E)))}\\ &\hspace{5cm}\notag\leq \norm{\tilde{b}}_{L^{p\maxsym 2}(0,t)}\lee\sqrt{t}+\norm{ Y}_{L^2(0,t;\Ep{p})}\p
< \infty \quad \textrm{a.s.,} 
\end{align}
where in the last inequality we use $\Bo{SL}_{\Ca{F}}^2(0,t;\Ep{p})\subset L^2((0,t)\times\Omega;\Ep{p})
$. Consequently, from \eqref{gammanorm} it follows that $u\mapsto \Ca{T}(t-u)G(Y(u))$ belongs to $\gamma(L^2(0,t,H);\Ep{p})$ a.s.

Having checked conditions \ref{F}, \ref{G} and that process $\Ca{T}(t-u)G(Y(u))$ is in $\gamma(L^2(0,t;H),E)$ a.s.\ we may apply Theorem \ref{t:varcons} to obtain the desired result.
\end{proof}
In the sequel we need Lemma \ref{Lemma 4.1} (cf. Lemma 4.1 in \cite{coxgorajski}). The proof of the following lemma is left to the reader. 
\begin{lemma}\label{Lemma 4.1}
Let $t > 0$, $p\in[1,\infty)$ and $g\in L^p(-1,t;E)$. For all $s\in[0,t]$ let us denote $y(s)=g_s$ the segments of function $g$. Then the following hold: 
\begin{enumerate}
    \item the functions $y:[0,t]\mapsto L^p(-1,0;E)$, $[0,t]\ni s \mapsto \int_0^s y(r) dr \in W^{1,p}([-1,0];E)$ are continuous, and \begin{align} \label{sup}\sup_{s\in[0,t]}\norm{y(s)}_{L^p(-1,0;E)}&\leq \norm{g}_{L^p(-1,t;E)},\\
\frac{d}{d\theta}\int_0^ty(s)ds&=y(t)-y(0),    
    \end{align}
for a.e. $\theta \in [-1,0]$.
	 \item for every $h \in L^{1}(0,t)$ the mapping $yh:[0,t]\to L^{p}(-1,0;E)$ is Bochner integrable and
\begin{equation}\label{r:B2}
    \lee\int_0^ty(s)h(s)ds\p(\theta)=\int_0^tg(s+\theta)h(s)ds,
\end{equation} for a.e. $\theta \in[-1,0]$.  
\end{enumerate}
\end{lemma}

From Lemma \ref{Lemma 4.1} we obtain the following:
\begin{lemma}\label{l:A2}
Let $E_1$ be a Banach space and let  $[-1,t]\ni s \mapsto T(s)\in \Ca{L}(E_1,E)$ be a $E_1$-strongly measurable and uniformly bounded mapping, i.e. there exists $M>0$ such that for all $s\in [-1,t]$  
$$\n T(s)\n_{\Ca{L}(E_1,E)}\leq M.$$ For all $s\in[0,t]$ let us denote $T_s:[-1,0]\to \Ca{L}(E_1,E)$ the mapping given by $T_s(\theta):=T(s+\theta)$, $\theta \in[-1,0]$. 
Then for any strongly measurable $f:(0,t)\to E_1$ such that $\int_0^t\lee\int_{-1}^0 \len T(s+\theta)f(s)\pn^p d\theta\p^\inv{p}ds<\infty$ for some $p\geq 1$ the mapping $[0,t]\ni s \mapsto T_sf(s)\in L^p(-1,0;E)$, $T_sf(s)(\theta):=T(s+\theta)f(s)$, $\theta \in[-1,0]$ is Bochner integrable on $[0,t]$ and for a.e. $\theta\in[-1,0]$ 
\begin{align}\label{A4}
\lee\int_0^tT_sf(s)ds\p(\theta)=\int_0^tT(s+\theta)f(s)ds.
\end{align} 
\end{lemma}

\begin{remark}\label{r:pi2YinLp}

Let $Y$ be a mild solution to \eqref{SDCP}. By Lemma \ref{l:A2} (applied to $E_1=\Ep{p}$, $T(u)=1_{u\geq 0}\pi_1\Ca{T}(u)$, $u\in[-1,t]$ and $f=F(Y)$) and by \eqref{propT}  it follows that for a.e. $\theta\in[-1,0]$ 
\begin{align}\label{r:pi2Bochint}
\lee\int_{0}^{t} \pi_2 \Ca{T}(t-s)F(Y(s)) ds\p(\theta) = \int_{0}^{t+\theta} \pi_1 \Ca{T}(t-s+\theta)F(Y(s))ds \textrm{
a.s.}
\end{align}
Let $p'$ be such that $\inv{p}+\inv{p'}=1$. By testing the stochastic convolution in equation \eqref{Dvarcon} against elements of $E^*\times L^{p'}(-1,0;E^*)$, which is norming for $\Ep{p}$ (cf. chapter III in~\cite{Diestel}), and applying equality \eqref{propT} one can shows that for a.e.   $\theta\in[-1,0]$ 
\begin{align} \label{r:pi2stochint}
\lee\int_{0}^{t} \pi_2 \Ca{T}(t-s)G(Y(s)) dW_H(s)\p(\theta) =\int_{0}^{t+\theta} \pi_1 \Ca{T}(t-s+\theta)G(Y(s))dW_H(s)  \ \textrm{a.s.}
\end{align}
Hence by Theorem \ref{t:Dvarcons} and by \eqref{propT}, \eqref{r:pi2Bochint} and \eqref{r:pi2stochint} it follows that if the assumptions of Theorem \ref{t:Dvarcons} holds and $Y$ is a weak solution to \eqref{SDCP}, then for all $t\geq 0$ 
\begin{align}\label{pi2Y}\pi_2 Y(t)(\theta) = \eqtwo{\pi_1 Y(t+\theta)}{t+\theta\geq 0}{f(t+\theta)}{t+\theta<0} \text{a.s.,}
\end{align} in particular it follows that $\pi_1 Y\in L_{loc}^p(0,\infty;E)$ a.s.

\end{remark}
\begin{proof}[The proof of \eqref{r:pi2stochint}] 
 Since $Y$ is strongly measurable there is no loss of generality in assuming that $E$ is separable. Then,   $L^{p}(-1,0;E)$ is also separable and there exists countable set $N^*\subset L^{p'}(-1,0;E^*)$ which is norming for  $L^{p}(-1,0;E)$ (cf. Lemma 1.1 in~\cite{vanNeervenISEM}). Fix $f^*\in N^*$, then by \eqref{propT} the stochastic  Fubini theorem in Hilbert space (cf. Theorem 3.5 in~\cite{vanNeervenVeraar_Fub}) we obtain
\begin{align*}
&\la \int_{0}^{t} \pi_2 \Ca{T}(t-s)G(Y(s))dW_H(s), f^* \ra_{L^p(-1,0;E),L^{p'}(-1,0;E^*)}\\&\quad= \int_{0}^{t} (\pi_2 \Ca{T}(t-s)G(Y(s)))^*f^* dW_H(s)\\
&\quad=  \int_{0}^{t} \int^0_{-1}1_{t-s+\theta\geq 0}G^*(Y(s))(\pi_1\Ca{T}(t-s+\theta)^*\vectwo{f^*(\theta)}{0}d\theta dW_H(s)\\&\quad=\int^0_{-1}\int_0^t 1_{t-s+\theta\geq 0}G^*(Y(s))(\pi_1\Ca{T}(t-s+\theta))^*\vectwo{f^*(\theta)}{0}dW_H(s)d\theta\\
&\quad=\int^0_{-1}\la\int_0^t 1_{t-s+\theta\geq 0}\pi_1\Ca{T}(t-s+\theta)G(Y(s))dW_H(s),f^*(\theta)\ra_{E,E^*} d\theta\\
&\quad=\la \theta\mapsto\int_0^{t+\theta} \pi_1\Ca{T}(t-s+\theta)G(Y(s))dW_H(s),f^*\ra_{L^p(-1,0;E),L^{p'}(-1,0;E^*)} \quad \textrm{a.s.}
\end{align*}
\end{proof}
 
 Let us introduce the following hypotheses:
\renewcommand\theenumi {\arabic{enumi}}
\renewcommand\theenumii {\alph{enumii}}
\renewcommand\labelenumi{(H\theenumi)}
\labelformat{enumi}{(H#1)}

\begin{enumerate} 
\setcounter{enumi}{4}
\item \label{h5} for all $t>0$ and all  $f\in L^1(0,t;\Ep{p})$ the function $\phi(f)$ is Bochner integrable on $[0,t]$.

\item \label{h6} for all $t>0$ and all $f\in L^2(0,t;\Ep{p})$ the operator $\psi(f)$ represents element from $\gamma(L^2(0,t;H),E)$.
\end{enumerate}   

Adding hypotheses  \ref{h5}, \ref{h6} to the set of assumptions of Theorem \ref{t:varcons2} we can establish the following result which generalised Theorem 4.2 from \cite{coxgorajski}: 
\renewcommand\theenumi {\roman{enumi}}
\renewcommand\labelenumi{(\theenumi) }
\begin{thm}\label{t:Dvarcons2}
Under the hypotheses of Theorem \ref{t:Dvarcons} and assuming additionally  \ref{h5}, \ref{h6} the following conditions are equivalent:  
\begin{enumerate}
    \item $Y$ is a generalised strong solution to \eqref{SDCP},
	 \item $Y$ is a weak solution to \eqref{SDCP},
	 \item $Y$ is a mild solution to \eqref{SDCP}.
\end{enumerate}
\end{thm}
\begin{proof}
We apply Theorem \ref{t:varcons2} and use the proof of Theorem \ref{t:Dvarcons}. First notice that assumptions \ref{F'} and \ref{G} hold for $F=[\phi,0]'$ and $G=[\psi,0]'$, by \ref{h5}, \ref{h3}  and by \ref{h4}, respectively (see also Lemma \ref{l:eqSolSCPD}). Let $Y\in \Bo{SL}_{\Ca{F}}^2(0,t;\Ep{p})$ for some $t>0$.  By \ref{h6} the process $G(Y)=[\psi(Y),0]'$ is in $\gamma(L^2(0,t;H),E)$ a.s. .   
We shall now prove that $u\mapsto\int_{0}^{t-u} \Ca{T}(s)G(Y(u)) ds $ belongs to $\gamma(L^2(0,t;H),E)$ a.s. Remark \ref{r:varcons}.(iii) yields that for every $u<t$ and all $s\in[0,t]$ the integral  $\int_{0}^{t-u} \Ca{T}(s)G(Y(u)) ds $ takes values in $\gamma(H,\Ep{p})$. Set $M_{\Ca{T}}(t):=\sup_{u\in[0,t]}\n \mathcal{T}(u)\n_{\Ca{L}(\Ep{p})}$. From Lemma \ref{l:eqSolSCPD} we obtain
\begin{align*}
\norm{ u\mapsto \pi_1 \int_{0}^{t-u} \Ca{T}(s)G(Y(u))ds}&_{L^2(0,t;\gamma(H,E))}
\\& \leq\lee\int_0^t\lee\int_0^{t-u}\tilde{b}(s)(1+\norm{ Y(u)}_{\Ep{p}})ds\p^2du \p^\inv{2}\\
&\leq\int_0^t\tilde{b}(s)ds \lee\sqrt{t} + \norm{Y}_{L^2(0,t;\Ep{p})}\p<\infty \ \textrm{a.s.}
\end{align*}
Hence and using \eqref{propT} and Lemma \ref{Lemma 4.1} we conclude that  
\begin{align*}
&\norm{ (\theta,u)\mapsto \left(\pi_2\int_{0}^{t-u} \Ca{T}(s)G(Y(u))ds\right)(\theta)}_{L^p(-1,0; L^2(0,t;\gamma(H,E))}  \\
&\hspace{2cm} = \left( \int_{-1}^{0}\left(\int_0^t \norm{ \int_{0}^{t-u} 1_{\{s+\theta\geq 0\}}\pi_1\Ca{T}(s+\theta)G(Y(u))ds}_{\gamma(H,E)}^2du\right)^\frac{p}{2} d\theta \right)^{\inv{p}}\\
&\hspace{2cm} = \lee \int_{-1}^{0} \norm{ u\mapsto\int_{0}^{t-u+\theta} \pi_1\Ca{T}(s)G(Y(u))ds}_{L^2(0,t;\gamma(H,E))}^{p} d\theta \p^{\inv{p}}\\
&\hspace{2cm} \leq\int_0^t\tilde{b}(s)ds \lee\sqrt{t} + \norm{ Y}_{L^2(0,t;\Ep{p})}\p<\infty \quad \textrm{a.s.}
\end{align*}
\end{proof}
\subsubsection{Existence and uniqueness of mild solution to \eqref{SDCP}}
Existence and uniqueness of mild solution to \eqref{SDCP} in the case $\phi=0$ and under the Lipschitz-continuity assumption on $\psi$ is proved in \cite[Theorem 4.4]{coxgorajski}.  
\begin{thm}\label{t:existSDCP}
Let the assumptions of Theorem \ref{t:Dvarcons} hold. In addition, assume that $Y_0:=[x_0,f_0]'\in L^q((\Omega,\F_0),\Ep{p}))$ for some $q \in [2,\infty)$.
Then for every $t>0$  there exists a unique process $Y\in \Bo{SL}_{\Ca{F}}^q(0,t;\Ep{p})$ for which \eqref{Dvarcon}
holds. Moreover, there exists $L>0$ such that for all
$\Ca{X},\Ca{Y}\in L^{q}(\Omega;\Ep{p})$
\begin{align}\label{exist_voc2_1}
&\sup_{s\in[0,t]}\E \ \norm{Y(s;\Ca{X})}_{\Ep{p}}^{q}\leq L \Big(1+ \E\ \norm{\Ca{X}}_{\Ep{p}}^{q}\Big),\\
&\sup_{s\in[0,t]}\E \ \norm{Y(s;\Ca{X})-Y(s;\Ca{Y})}_{\Ep{p}}^{q}\leq L\ \E\ \norm{\Ca{X}-\Ca{Y}}_{\Ep{p}}^{q}, \label{exist_voc2_2}
\end{align}
 \end{thm}
\begin{proof}
Let us fix $q\in[2,\infty)$ and $t>0$. In the Banach space $\Bo{SL}_{\Ca{F}}^{q}(0,t;\Ep{p})$ we introduce an equivalent norm
\begin{align}
\norm{Y}_{\beta}=\lee\sup_{s\in[0,t]}e^{-\beta s}\E\ \norm{Y(s)}_{\Ep{p}}^{q}\p^\inv{q}, \quad \beta \geq 0. 
\end{align}
 Let us define $\Ca{K}:\Bo{SL}_{\Ca{F}}^{q}(0,t;\Ep{p})\to \Bo{SL}_{\Ca{F}}^{q}(0,t;\Ep{p})$ as follows
$$\mathcal{K}(Z)(s)=\Ca{T}(s)\vectwo{x}{f}+\int_0^s\Ca{T}(s-u)F(Z(u))du+\int_0^s \Ca{T}(s-u)G(Z(u))dW_H(u)$$ for all $s\in [0,t]$.
From the proof of Theorem \ref{t:Dvarcons} and Lemma \ref{l:eqSolSCPD} it follows that $\Ca{K}$ is well defined.
To prove that $\Ca{K}(Z)$ is indeed in~$\Bo{SL}_{\Ca{F}}^{q}(0,t;\Ep{p})$, observe that by the Burkholder-Davis-Gundy inequality \eqref{BDG}, and then repeating the steps from the first part of the proof of Theorem \ref{t:Dvarcons} (cf. \eqref{gammanorm}-\eqref{pi2a_2}) we see that for every $s\in[0,t]$
\begin{align*}
&\lee \E\ \norm{\int_0^s \Ca{T}(s-u)G(Z(u))dW_H(u)}_{\Ep{p}}^q\p^{\inv{q}} \\ &\hspace{2.5cm}\lesssim_q \lee\E\ \norm{ u\mapsto\Ca{T}(s-u)G(Z(u))}^q_{\gamma(0,s;H,\Ep{p})}\p^\inv{q} \\
&\hspace{2.5cm}\lesssim_q^{\eqref{gammanorm}} \lee\E\ \norm{ u\mapsto\pi_1\Ca{T}(s-u)G(Z(u))}^{q}_{L^2(0,s;\gamma(H,E))}\p^\inv{q}\\
&\hspace{2.5cm}\quad \quad +\lee\E\ \norm{(\theta,u)\mapsto\pi_2\Ca{T}(s-u)G(Z(u))}^{q}_{L^p(-1,0;L^2(0,s;\gamma(H,E)))}\p^\inv{q}\\
&\hspace{2.5cm}\leq^{\eqref{pi1a}-\eqref{pi2a_2}} \lee\E\ \lee\int_0^s \tilde{b}^2(s-u)\lee 1+\norm{ Z(u)}_{\Ep{p}}\p^{2}du\p^\frac{q}{2}\p^\inv{q}\\&\hspace{2.5cm}\quad+\lee\E\ \lee\int_0^s\norm{\tilde{b}}_{L^{p\maxsym 2}(0,s)}^2(1+\norm{ Z(u)}_{\Ep{p}})^2du\p^\frac{q}{2}\p^{\inv{q}}\\
&\hspace{2.5cm}\leq \lee\int_0^s \tilde{b}^2(s-u)e^{2\beta u}e^{-2\beta u}\lee\E\ (1+\norm{ Z(u)}_{\Ep{p}})^{q}\p^\frac{2}{q}du\p^\inv{2}\\&\hspace{2.5cm}\quad+\norm{\tilde{b}}_{L^{p\maxsym 2}(0,s)}\lee\int_0^se^{2\beta u}e^{-2\beta u}\lee\E\ (1+\norm{ Z(u)}_{\Ep{p}})^q\p^\frac{2}{q}du\p^{\inv{2}}\\
&\hspace{2.5cm}\leq \lee\lee\int_0^s\tilde{b}^2(s-u)e^{2\beta u}du\p^\inv{2}+\norm{ \tilde{b}}_{L^{p\maxsym 2}(0,s)}\lee \int_0^s e^{2\beta u}du\p^\inv{2}\p\lee 1+\norm{ Z }_{\beta}\p,
\end{align*}
where in second to last inequality we use  the Minkowski's integral inequality.
Hence
\begin{align}\label{ist:stoch1}
&\sup_{s\in[0,t]}e^{-\beta s}\lee \E\ \norm{\int_0^s \Ca{T}(s-u)G(Z(u))dW_H(u)}^q\p^{\inv{q}}\\
&\leq\lee 1+\norm{ Z}_{\beta}\p\sup_{s\in[0,t]}\lee\lee\int_0^s\tilde{b}^2(s-u)e^{-2\beta(s- u)}du\p^\inv{2}+\norm{ \tilde{b}}_{L^{p\maxsym 2}(0,s)}\lee\int_0^se^{-2\beta(s- u)}du\p^\inv{2}\p\notag\\
&\quad\quad\quad\leq(1+\norm{ Z }_{\beta})C_{\beta,b},\notag
\end{align}
where $C_{\beta,b}=\lee\int_0^t\tilde{b}^2(u)e^{-2\beta u}du\p^\inv{2}+\norm{ \tilde{b}}_{L^{p\maxsym 2}(0,t)}\sqrt{\inv{2\beta}\lee 1-e^{-2\beta}\p}$.\par
Similarly, for all $Z_1, Z_2\in \Bo{SL}_{\Ca{F}}^{q}(0,t;\Ep{p})$ we obtain:
\begin{align}\label{ist:stoch2}
\sup_{s\in[0,t]}e^{-\beta s}\lee \E\ \norm{\int_0^s \Ca{T}(s-u)(G(Z_1(u))-G(Z_2(u)))dW_H(u)}^q\p^{\inv{q}}
\leq C_{\beta,b}\norm{ Z_1-Z_2 }_{\beta}.
\end{align}

By \eqref{propT}, \eqref{r:DFa} and the Minkowski's integral inequality for all $Z\in \Bo{SL}_{\Ca{F}}^{q}(0,t;\Ep{p})$ and for every $s\in[0,t]$ one gets   
\begin{align*}
e^{-s\beta}&\lee\E\ \norm{\int_0^s \Ca{T}(s-u)F(Z(u))du}^q\p^\inv{q}\leq \notag\\
&\leq e^{-s\beta}\int_0^s \lee \tilde{a}(s-u)e^{\beta u}+\len\tilde{a}\pn_{L^p(0,s)}e^{\beta u}\p e^{-\beta u}\lee\E\ \lee 1+\norm{ Z(u)}_{\Ep{p}}\p^q\p^\inv{q}du \notag\\
 \notag
&\leq\lee 1+\n Z \n_{\beta}\p\lee \int_0^s \tilde{a}(u)e^{-\beta u}du+\len\tilde{a}\pn_{L^p(0,s)}\inv{\beta}(1-e^{-\beta})\p.
\end{align*}
Hence
\begin{align}
\label{ist:bochner1}
\sup_{s\in[0,t]}e^{-s\beta}\lee\E\ \norm{\int_0^s \Ca{T}(s-u)F(Z(u))du}^q\p^\inv{q}\leq C_{\beta,a}\lee 1+\norm{ Z }_{\beta}\p, 
\end{align}
where $C_{\beta,a}=\int_0^t \tilde{a}(u)e^{-\beta u}du+\len\tilde{a}\pn_{L^p(0,s)}\inv{\beta}(1-e^{-\beta})$.
Between the same lines using \eqref{r:DF} for all $Z_1,Z_2\in \Bo{SL}_{\Ca{F}}^{q}(0,t;\Ep{p})$ one has
\begin{align}\label{ist:bochner2}
\sup_{s\in[0,t]}e^{-s\beta}\lee\E\ \norm{\int_0^s \Ca{T}(s-u)\lee F(Z_1(u))-F(Z_1(u))\p du}^q\p^\inv{q}\leq C_{\beta,a}\norm{ Z_1 -Z_2}_{\beta}. 
\end{align}
Finally, by \eqref{ist:stoch1}-\eqref{ist:stoch2} and by \eqref{ist:bochner1}-\eqref{ist:bochner2} we have
\begin{align}\label{K1}
&\norm{\mathcal{K}(Z_1)}_{\beta}\leq M_{\Ca{T},t}\norm{[x,f]^{'}}_{L^q(\Omega;\Ep{p})}+C_{E,q}K_\beta\lee 1+\norm{Z_1}_{\beta}\p\\
&\norm{\mathcal{K}(Z_1)-\mathcal{K}(Z_2)}_{\beta} \leq C_{E,q} K_{\beta} \norm{Z_1-Z_2}_{\beta}, \label{K2}
\end{align}
for all $Z_1,Z_2\in \Bo{SL}_{\Ca{F}}^{q}(0,t;\Ep{p})$, where $M_{\Ca{T},t}=\sup_{u\in[0,t]}\norm{\Ca{T}(u)}_{\Ca{L}(\Ep{p})}$, $$K_{\beta}=C_{\beta,a}+C_{\beta,b}$$ and $C_{E,q}$ is a constant depending on $q$ and the space $E$ and equal to a product of constant in   the Burkholder-Davis-Gundy inequality \eqref{BDG} and the constant from the type  2 property of $E$.   
For $\beta$ such that $K_{\beta}<\inv{C_{E,q}}$ the mapping $\mathcal{K}$ is a strict contraction. Then, by the Banach fix point theorem we have the existence and  of mild solution to \eqref{SDCP}.

To prove inequalities \eqref{exist_voc2_1}-\eqref{exist_voc2_2} we fix $\Ca{X},\Ca{Y}\in L^{q}(\Omega;\Ep{p})$. By \eqref{K2}  we obtain 
\begin{align*}
     \norm{Y(\cdot;\Ca{X})-Y(\cdot;\Ca{Y})}_{\beta}&\leq \n\Ca{T}(\cdot)(\Ca{X}-\Ca{Y})\n_\beta+\n\Ca{K}(Y(\cdot;\Ca{X}))-\Ca{K}(Y(\cdot;\Ca{Y}))\n_\beta\\&\leq M_{\Ca{T},t} \norm{\Ca{X}-\Ca{Y}}_{L^q(\Omega;\Ep{p})}+ C_{E,q}K_{\beta}\norm{Y(s;\Ca{X})-Y(s;\Ca{Y})}_{\beta}.
\end{align*} 
We take  $\beta>0$ such that $K_{\beta}<\inv{2C_{E,q}}$. Then from the above inequality we conclude that
\begin{align*}
\norm{Y(\cdot;\Ca{X})-Y(\cdot;\Ca{Y})}_{\beta}\leq 2M_{\Ca{T},t} \norm{\Ca{X}-\Ca{Y}}_{L^q(\Omega;\Ep{p})}.
\end{align*}
Similarly, by \eqref{K1} it follows that:
\begin{align*}
\norm{Y(;\Ca{X})}_\beta\leq 4M_{\Ca{T},t}\norm{\Ca{X}}_{L^q(\Omega;\Ep{p})}+1.
\end{align*}\par

\end{proof}

\subsection{Equivalence of solutions to \eqref{SDE} and \eqref{SDCP}}
Consider the problem \eqref{SDE} as given in the introduction with a fixed $p\in [1,\infty)$. 
\begin{defn}\label{d:slaberozw_SRzO}
A strongly measurable adapted process $X:[-1,\infty)\times\Omega \rightarrow E$ \footnote{For all $t\in[-1,0]$ we assume that $X(t)$  is a $\Ca{F}_0$-strongly mesurable.} is called a \emph{weak solution} to \eqref{SDE} if $X$ belongs to $L_{loc}^p(0,\infty;E)$ a.s. and for
all $t> 0$ and $x^*\in D(B^\odot)$:
 \begin{enumerate}
\item $s\mapsto\langle \phi(X(s),X_s), x^*\rangle$ is integrable on $[0,t]$ a.s.; 
 \item $(s,\omega)\mapsto\psi^*(X(s),X_s)x^*$ is stochastically integrable on $[0,t]$;
\item  almost surely 
\begin{align}\notag 
\langle X(t), x^*\rangle -\langle x_0,x^* \rangle&=\langle \Phi\int_{0}^{t} X_sds, x^*\rangle +\int_0^t\langle X(s), B^\odot x^*\rangle ds \\  
&+\int_0^t\langle \phi(X(s),X_s), x^*\rangle ds
+\int_0^t\psi^*(X(s),X_s)x^*dW_H(s)
\label{weaksolSDE}
\end{align}
\item $X_0=f_0$.
\end{enumerate}
\end{defn}
We say that $X$ is a  generalised strong solution to \eqref{SDE} if $X$ is a weak solution to \eqref{SDE} such that for all $t>0$  \begin{align}\label{gstrongsolSDE_1} \int_0^tX(s)ds&\in D(B) \text{ a.s., and} \\ X(t) - x_0 &= \Phi\int_{0}^{t} X_sds +B\int_0^t X(s)ds +\int_0^t \phi(X(s),X_s) ds
+\int_0^t\psi(X(s),X_s)dW_H(s) \text{ a.s.}.
\label{gstrongsolSDE_2}
\end{align}
\begin{thm} \label{t:rep}
Under hypotheses of Theorem \ref{t:Dvarcons} the following conditions hold:
\begin{enumerate}
\item Let $X$ be a weak solution to \eqref{SDE}, then the process $Y$ defined by $Y(t):=[X(t),X_t]'$ is a weak solution to \eqref{SDCP}. \label{repr1}
\item On the other hand, if $Y$ is a weak solution to \eqref{SDCP} then the process defined by $X|_{[-1,0)}=f_0$, $X(t):=\pi_1Y(t)$ for $t\geq 0$ is a weak solution to \eqref{SDE}. \label{repr2}
\end{enumerate}

\end{thm}
\begin{proof}
$(i)$.
Let us fix $t>0$. Since $X\in L_{loc}^p(0,\infty;E)$ a.s., by Lemma \ref{Lemma 4.1} it follows that the process $Y=[X,X_s]^{'}$ takes values in~$\Ep{p}$ and is strongly measurable, adapted and Bochner integrable on $[0,t]$ a.s. Moreover, $\int_0^tX_sds\in W^{1,p}([-1,0];E)$, $\int_0^tX_sds(0)=\int_0^tX(s)ds$ a.s. Thus  
\begin{align}\label{Y}
&\int_0^t Y(s)ds=\vectwo{\int_0^t X(s)ds}{\int_0^tX_sds} \in E\times W^{1,p}([-1,0];E) \text{ a.s.}
\end{align}
Hence and by  \eqref{weaksolSDE} and by Lemma \ref{Lemma 4.1} for all $[y^
*,g^*]^{'}\in D(\Ca{A}^\odot)$ the following equality holds for almost all $\omega$:
\begin{align}\notag
&\la \int_0^t Y(s)ds,\Ca{A}^\odot[y^*,g^*]'\ra\\&\notag\hspace{3cm}=\la \int_0^tX(s)ds, B^\odot y^*\ra+\la \Phi\int_0^t X_sds, y^*\ra +\la\frac{d}{d\theta}\int_0^tX_sds,g^*\ra\\&
\hspace{3cm}=\la X(t)-x,y^*\ra-\int_0^t\la\phi(X(s),X_s),y^*\ra ds-\int_0^t \psi^*(X(s),X_s)y^* dW_H(s) \notag\\ 
&\hspace{4cm}+\la X_t-f,g^*\ra. \label{Y2}
\end{align}
The sum of the first and last terms on right hand side of \eqref{Y2}  is equal to   
$\la Y(t)-[x,f]^{'},[y^*,g^*]^{'}\ra$. Adding the second and the third term on right hand side of \eqref{Y2}, we obtain $$-\int_0^t\la F(Y(s)), [y^*,g^*]^{'}\ra ds -\int_0^tG^*(Y(s))[y^*,g^*]^{'}dW_H(s).$$ Finally, by inserting the above sums into \eqref{Y2} we obtain, almost surely, 
\begin{align*}
\la \int_0^t Y(s)ds,\Ca{A}^\odot[y^*,g^*]'\ra&=\la Y(t)-[x,f]^{'},[y^*,g^*]^{'}\ra\\&-\int_0^t\la F(Y(s)), [y^*,g^*]'\ra ds-\int_0^tG^*(Y(s))[y^*,g^*]'dW_H(s).
\end{align*}

$(ii)$ By Theorem \ref{t:Dvarcons} the process $Y$ satisfies \eqref{voc2}. Remark \ref{r:pi2YinLp} leads to   
\begin{align} 
\label{r:rep0}
\pi_2 Y(t)(\theta) = \eqtwo{\pi_1 Y(t+\theta)}{t+\theta\geq 0}{f(t+\theta)}{ t+\theta< 0} \textrm{ a.s.}
\end{align}
for all $t>0$.
We conclude from Definition \ref{d:weaksolSCP} and then form Lemma \ref{Lemma 4.1} applied to $\pi_2Y$ that 
\begin{align}
\int_0^t\psi^*&(Y(s))y^*dW_H(s)=\int_0^tG^*(Y(s))[y^*,g^*]^{'}dW_H(s)\notag\\
&=
\la Y(t)-[x,f]^{'},[y^*,g^*]^{'}\ra-\la\int_0^t Y(s)ds,\Ca{A}^\odot[y^*,g^*]^{'}\ra-\int_0^t\la F(Y(s)),[y^*,g^*]^{'}\ra ds\notag\\
&=\la Y(t)-[x,f]^{'},[y^*,g^*]^{'}\ra-\la \pi_1\int_0^tY(s)ds,B^\odot y^*\ra-\la \Phi\int_0^t \pi_2Y(s)ds,y^*\ra\notag\\
&\quad -\la \frac{d}{d\theta}\int_0^t\pi_2Y(s)ds,g^*\ra -\int_0^t\la \psi(Y(s)),y^*\ra ds\notag\\
&=\la \pi_1Y(t)-x,y^*\ra-\la \int_0^t\pi_1Y(s)ds,B^\odot y^*\ra-\la \Phi\int_0^t \pi_2Y(s)ds,y^*\ra\notag
 \\&\quad\quad -\int_0^t\la\psi(Y(s)),y^*\ra ds \textrm{ a.s.}\label{r:rep1},
\end{align}
for every $[y^*,g^*]^{'}\in D(\Ca{A}^\odot)$, where in the last equality we use the following identity for a.e. $\theta \in[-1,0]$:  $\frac{d}{d\theta}\int_0^t\pi_2Y(s)ds=\pi_2Y(t)-f$   (cf. \eqref{r:rep0} i Lemma \ref{Lemma 4.1}).\par
Let us denote $$X(s)=\eqtwo{\pi_1Y(s)}{s\geq 0}{f(s)}{s\in(-1,0)}.$$  Since $Y$ has Bochner integrable trajectories a.s. and satisfies \eqref{r:rep0}, we have $ \int_{0}^{t}  |X (s)|^{p} ds < \infty$ a.s.  
By
 \eqref{r:rep0} and again by Lemma \ref{Lemma 4.1} we obtain
\begin{align}\label{r:rep11} \lee\int_0^t\pi_2Y(s)ds\p(\theta)=\lee\int_0^tX_sds\p(\theta) \quad \textrm{ dla p.w. }\theta \in[-1,0].
\end{align}
On substituting \eqref{r:rep11} into \eqref{r:rep1} we finally obtain, almost surely,
\begin{align*}
\int_0^t\psi^*(X(s),X_s)y^*dW_H(s)&=\la X(s)-a,y^*\ra-\la \int_0^tX(s)ds,B^\odot y^*\ra\\&\qquad-\la\Phi\int_0^t X_sds,y^*\ra -\int_0^t\la \psi(X(s),X_s),y^*\ra ds.
\end{align*}
\end{proof}
Theorems  \ref{t:Dvarcons2} and \ref{t:rep} combined give the following theorem which has been also established   in \cite[Theorem 4.8]{coxgorajski} under stronger set of assumptions. 
\begin{thm} \label{t:rep2}
Under hypotheses of Theorem \ref{t:Dvarcons2} the following conditions hold:
\begin{enumerate}
\item Let $X$ be a generalised strong solution to \eqref{SDE}, then the process $Y$ defined by $Y(t):=[X(t),X_t]'$ is a generalised strong solution to \eqref{SDCP}. 
\item On the other hand, if $Y$ is a generalised strong solution to \eqref{SDCP}, then the process defined by $X|_{[-1,0)}=f_0$, $X(t):=\pi_1(Y(t))$ for $t\geq 0$ is generalised strong solution to \eqref{SDE}. 
\end{enumerate}
\end{thm}
\begin{corol}\label{c:weak_eq_strongSDE}
Assume that the assumptions of Theorem \ref{t:Dvarcons2} hold. Then $X$ is a weak solution to \eqref{SDE} if and only if $X$ is a generalised strong solution to \eqref{SDE}.
\end{corol} 
\begin{corol}\label{c:existSDE}
Let the hypotheses of Theorem \ref{t:Dvarcons} hold and assume that $[x,f]\in L^{p\maxsym q}_{\Ca{F}_0}(\Omega;\Ep{p})$ for some $p\geq 1$ and $q\geq 2$. Then, for all $t>0$

\begin{enumerate}
\item in the space  $\Bo{SL}_{\Ca{F}}^{q\maxsym p}(0,t;E)$ there exists a unique weak solution $X=X(\cdot,[x,f]')$ to \eqref{SDE}. Moreover, $X$ satisfies the equation, almost surely,
\begin{align}\label{SDEsol}
X(t) = \pi_1\Ca{T}(t)[x,f]'+\int_{0}^{t} \pi_1\Ca{T}(t-s)[\phi(X(s)),0]'ds+ \int_{0}^{t} \pi_1\Ca{T}(t-s)[\psi(X(s)),0]'dW_H(s).  
\end{align}
\item  there exists $L>0$ such that
$\Ca{X},\Ca{Y}\in L^{q\maxsym p}(\Omega;\Ep{p})$
\begin{align}\label{SDEsol_ini}
&\sup_{s\in[0,t]}\E \ \norm{X(s;\Ca{X})}_{E}^{q\maxsym p}\leq L \big(1+ \E\ \norm{\Ca{Y}}^{q\maxsym p}_{L^{q\maxsym p}(\Omega;\Ep{p})}\big),\\
&\sup_{s\in[0,t]}\E \ \norm{X(s;\Ca{X})-X(s;\Ca{Y})}_{E}^{q\maxsym p}\leq L\ \E\ \norm{\Ca{X}-\Ca{Y}}_{L^{q\maxsym p}(\Omega;\Ep{p})}^{q\maxsym p};\label{r:SRzOwarpocz2}
\end{align} 
\end{enumerate}  
\end{corol}
\begin{proof}
Let us fix $t>0$. By Theorems \ref{t:rep}, \ref{t:Dvarcons} and \ref{t:existSDCP} it is enough to show that if $X\in \Bo{SL}^{p\maxsym q}_{\Ca{F}}(0,t;E)$, then the segment process $(X_s)_{s\geq 0}$ belongs to  $\Bo{SL}_{\Ca{F}}^{q\maxsym p}(0,t;L^p(-1,0;E))$. Indeed, let us notice that 
\begin{align*}%\label{sl=os_1}
\sup_{s\in[0,t]}\n X_s\n_{L^p(-1,0;E)}\leq \n f\n_{L^p(-1,0;E)}+\n X\n_{L^p(0,t;E)} \textrm{ a.s.},
\end{align*}
Hence and by  Minkowski's inequality we obtain
\begin{align*}%\label{sl=os_1}
\sup_{s\in[0,t]}\norm{ X_s}_{L^{q\maxsym p}(\Omega;L^p(-1,0;E))}&\leq \norm{ f}_{L^{q\maxsym p}(\Omega;L^p(-1,0;E))} +\norm{X}_{L^{q\maxsym p}(\Omega;L^p(0,t;E)}\\&\leq \norm{ f}_{L^{q\maxsym p}(\Omega;L^p(-1,0;E))} +\norm{X}_{L^p(0,t;L^{q\maxsym p}(\Omega;E))} 
\\&\leq \norm{ f}_{L^{q\maxsym p}(\Omega;L^p(-1,0;E))} +t\norm{X}_{\Bo{SL}^{p\maxsym q}_{\Ca{F}}(0,t;E)}<\infty. 
\end{align*}
\end{proof}
Therefore, and by Theorems \ref{t:rep} and \ref{t:Dvarcons} and by Theorem 4.5  in \cite{coxgorajski} we obtain the corollary:
\begin{corol}\label{c:continuitySDE}
Under hypotheses of Theorem \ref{t:Dvarcons} and assumption  that $[x,f]'\in L^{q\maxsym p}(\Omega;\Ep{p})$ for some $p\geq 1$ and some $q> 2$, if there exists $\alpha  \in(\inv{q\maxsym p},\inv{2})$ such that for all $t>0$
\begin{align}\label{z:reg_delay}
\int_0^t(a(s)s^{-\alpha}+b^{2\maxsym p}(s)s^{-(2\maxsym p)\alpha})ds< \infty,
\end{align}
where $a$ and $b$ are the functions form the hypotheses \ref{h3} and \ref{h4}, respectively, then the weak solution  $X(\cdot,[x,f]')\in \Bo{SL}_{\Ca{F}}^{q\maxsym p}(0,t;E)$ to \eqref{SDE} belongs to  $L^{q\maxsym p}(\Omega;C([0,t];E)$. Moreover, there exists $L>0$ such that for all
$\Ca{X},\Ca{Y}\in L^{q\maxsym p}(\Omega;\Ep{p})$
\begin{align}\label{reg_mild_1_delay}
&\E \ \sup_{s\in[0,t]}\norm{X(s;\Ca{X})}_{E}^{q\maxsym p}\leq L \big(1+ \E\ \norm{\Ca{X}}_{L^{q\maxsym p}(\Omega;\Ep{p})}^{q\maxsym p}\big),\\
&\E \ \sup_{s\in[0,t]}\norm{X(s;\Ca{X})-X(s;\Ca{Y})}_{E}^{q\maxsym p}\leq L\E\ \norm{\Ca{X}-\Ca{Y}}_{L^{q\maxsym p}(\Omega;\Ep{p})}^{q\maxsym p}. \label{reg_mild_2_delay}
\end{align}
\end{corol}
 
\subsection{Examples}
\subsubsection{Stochastic reaction-diffusion equation with bounded delay}
Consider the following stochastic reaction-diffusion equation with bounded delay:
\begin{align}\label{ex}
\left\{\begin{array}{l} 
dy(t,\xi)=\Delta y(t,\xi)dt+\Big[\int_{t-1}^{t}\varphi(s-t,\xi)y(s,\xi)ds\\\quad+f_1(y(t,\xi))+ \int_{t-1}^{t}k_1(s-t,\xi)f_2(y(s,\xi))ds\Big]dt\\\quad+\left(g_1(y(t,\xi))+\int_{t-1}^{t}k_2(s-t,\xi)g_2(y(s,\xi))ds\right)d\mathcal{W}(t,\xi),
\\
y(t,0)=0,\quad y(t,1)=0,\\
y(0,\xi)=x_0(\xi),\ x_0\in L^r(0,1),\quad  y(\theta,\xi)=f_0(\theta,\xi), f_0\in L^p(-1,0;L^r(0,1)),  
\end{array}\right.
\end{align}  
where $\Delta=\frac{d^2}{d\xi^2}$,  $\varphi,k_1,k_2\in L^{p'}(-1,0;L^{\infty}(0,1))$ for some $p'\in(1,\infty]$ such that $\inv{p}+\inv{p'}=1$, and $f_1,f_2,g_1,g_2:\R\to\R$ are Lipschitz functions, and $\mathcal{W}(t,\xi)$ is a time-space Brownian motion on $[0,\infty)\times[0,1]$ and $r> 2$.
 This equation can be used to model phenomena in population dynamics (see \cite{baker}, \cite{Bocharov2000183} and \cite{kot2001elements}) or in neutral networks (see \cite{Lv20081590} and \cite{Bierkens2011}).  Let $B$ be a one-dimensional Laplacian on $E=L^r(0,1)$ with Dirichlet boundary conditions: 
\begin{align*}
B=\Delta_r=\frac{d^2}{d\xi^2}, \ D(\Delta_r)=\{x\in W^{2,r}([0,1]):x(0)=x(1)=0\}.
\end{align*}
Then, by Theorems 1.3.5 and 1.4.1 (see also (1.9.1)) in~\cite{daviesheat}  it follows that for all $q\in[1,\infty)$ the operator $(\Delta_q, D(\Delta_q))$ generates strongly continuous contraction semigroup  $(S_q(t))_{t\geq 0}$ on $L^q(0,1)$
\begin{align}\label{r:S}
(S_q(t)x)(s)=\int_0^1k(t,s,\xi)x(\xi)d\xi, 
\end{align}
for all $x\in L^q(0,1)$ and all $s\in[0,1]$, where
\begin{align}\label{r:j¹droCiep³a}
0<k(t,s,\xi)=\sum_{n=1}^\infty e^{\lambda_n t}e_n(s)e_n(\xi)\leq \inv{\sqrt{4\pi t}}e^{\frac{-(s-\xi)^2}{4t}},
\end{align}
and $e_n(s)=\sqrt{2}\sin(\pi n s)$ is an eigenvector of the Laplacian $B$  corresponding to the eigenvalue  $\lambda_n=-\pi^2n^2 $, $n\geq 1$. 
 Moreover, by the Aronson inequality (cf. \cite{PeszZab})  for all $x\in L^q(0,1)$ and $t>0$ the function $\xi \mapsto (S_q(t)x)(\xi)$ is continuous on $[0,1]$ and by 
\eqref{r:j¹droCiep³a} and then by  H\"older's inequality we obtain   
\begin{align}\label{r:ultrakontrakcja}
\sup_{s\in[0,1]}\left|(S_q(t)x)(s)\right|&\leq \sup_{s\in[0,1]}\int_0^1{\lee 4\pi t\p^{-\inv{2}}e^{\frac{-(s-\xi)^2}{4t}}|x(\xi)|}d\xi\\ \notag &\leq\sup_{s\in[0,1]}\lee\int_{\R{}} \lee 4\pi t\p^{-\frac{q'}{2}}e^{\frac{-q'(s-\xi)^2}{4t}}d\xi\p^\inv{q'}\norm{ x}_{L^q(0,1)}\\&\leq C_q t^{\inv{2q'}-\inv{2}}\n x\n_{L^q(0,1)},\notag 
\end{align}
where $C_q=(q')^{-\inv{q'}}(4\pi )^{\inv{2q'}-\inv{2}}$ and 
$\inv{q}+\inv{q'}=1$.
Let us introduce the notation:
\begin{align}
(\Phi h)(\xi) &=\int_{-1}^{0}\varphi(\theta,\xi)h(\theta,\xi)d\theta,\\
\phi([x,h]')(\xi)&=f_1(x(\xi))+ \int_{-1}^{0}k_1(\theta,\xi)f_2(h(\theta,\xi))d\theta,\\
(\psi([x,h]')u)(\xi)&=g_1(x(\xi))u(\xi)+\int_{-1}^{0}k_2(\theta,\xi)g_2(h(\theta,\xi))d\theta u(\xi),\\
W_H(t)u&=\int_0^1\int_0^tu(\xi)\mathcal{W}(dt,d\xi),
\end{align}
for all $[x,h]'\in\Ep{p}=L^r(0,1)\times L^p(-1,0;L^r(0,1))$, and all $u\in H=L^2(0,1)$.
Then, we can rewrite \eqref{ex} in the form \eqref{SDE} and by the following proposition the assumptions \ref{h1}-\ref{h4} hold: 
\begin{prop}\label{p:exRD}
Let $r> 2$ and  $p\geq 1$ be such that $2\maxsym p<\frac{4r}{2+r}$  and $E=L^r(0,1)$, $H=L^2(0,1)$. Then the following statements hold:
\begin{enumerate}
\item $\Phi\in\Ca{L}(L^p(-1,0;E),E)$ and $\Phi$ satisfies \ref{h2} with $$\eta:[-1,0]\to\Ca{L}(L^r(0,1)),\ \eta(\theta)x(\xi)=\lee\int_{-1}^\theta \varphi(s,\xi)ds\p x(\xi).$$
\item $\phi$ satisfies \ref{h3}. 
\item $\psi$ satisfies \ref{h4}.
\item $W_H$ is a $H$-cylindrical Wiener process. 
\end{enumerate}
\end{prop}
\begin{proof}
By the Minkowski integral inequality and by the H\"older inequality we conclude that for all $h\in L^p(-1,0;L^r(0,1))$
\begin{align*}
\len \Phi h\pn_{L^r(0,1)}&\leq \int_{-1}^0\left(\int_0^1|\varphi(\theta,\xi)h(\theta,\xi)|^r d\xi\right)^\inv{r}d\theta\\ &\leq \operatorname{ess\ sup}_{\xi\in[0,1]} \len\varphi(\cdot,\xi)\pn_{L^{p'}(-1,0)}\len h\pn_{L^p(-1,0;L^r(0,1))}.
\end{align*}
Notice that $\phi$ is Lipschitz continuous. Indeed, we have 
\begin{align*}
&\len \phi([x_1,h_1])-\phi([x_2,h_2])\pn_{L^r(0,1)}\leq L_{f_1}\len x_1-x_2\pn_{L^r(0,1)}\\&\quad
\quad+\left(\int_0^1\left|\int_{-1}^0k_1(\theta,\xi)(f_2(h_1(\theta,\xi))-f_2(h_2(\theta,\xi)))d\theta\right|^r d\xi\right)^\inv{r}\\&\quad
\leq L_{f_1}\len x_1-x_2\pn_{L^r(0,1)}+L_{f_2}\operatorname{ess\ sup}_{\xi\in[0,1]} \len k_1(\cdot,\xi)\pn_{L^{p'}(-1,0)}\len h_1-h_2\pn_{L^p(-1,0;L^r(0,1))},
\end{align*}
for all $[x_1,h_1], [x_2,h_2]\in \Ep{p}$,
where $L_{f_1}, L_{f_2}$ are, respectively, the Lipschitz constant of $f_1$ and $f_2$. 

Now we prove (iii). Let $\psi=\psi_1+\psi_2$ and $\psi_1(x)u=g_1(x)u$ for all $x\in L^r(0,1)$ and $u\in L^2(0,1)$.  Notice that if $g_1, g_2$ are not  constant, then $\psi$ does not belong to $\Ca{L}(H,E)$. But,  using \eqref{r:ultrakontrakcja} with some $ \frac{p\maxsym 2}{2}< q<\frac{2r}{2+r}$ and then by H\"older's inequality  we see that for all $t>0$ and all $u\in L^2(0,1)$, $[x, h] \in \Ep{p}$
\begin{align}\label{ex_psi_1}
&\sup_{\xi\in[0,1]} |(S(t)\psi_1(x)u)(\xi)|\leq
C_{q}t^{\frac{q-1}{2q}-\inv{2}}\len \psi_1(x)u\pn_{L^q(0,1)}\\&\leq
C_{q}t^{\frac{q-1}{2q}-\inv{2}}\len g_1(x)\pn_{L^{\frac{2q}{2-q}}(0,1)}\len u\pn_{L^2(0,1)}\leq L_{g_1}C_{q}t^{\frac{q-1}{2q}-\inv{2}}(1+\len x\pn_{L^{r}(0,1)})\len u\pn_{L^2(0,1)},\notag
\end{align}
and
\begin{align}\label{ex_psi_2}
&\sup_{\xi\in[0,1]} \left|\left(S(t)\psi_2(h)u\right)(\xi)\right|\leq
C_{q}t^{\frac{q-1}{2q}-\inv{2}}\len \psi_2(h)u\pn_{L^q(0,1)}\\\notag&\leq
C_{q}t^{\frac{q-1}{2q}-\inv{2}}\left(\int_0^1\left|\int_{-1}^0k_2(\theta,\xi)g_2(h(\theta,\xi))d\theta\right|^{\frac{2q}{2-q}} d\xi\right)^{\frac{2-q}{2q}}\len u\pn_{L^2(0,1)}\\&\leq L_{g_2}\operatorname{ess\ sup}_{\xi\in[0,1]} \len k_2(\cdot,\xi)\pn_{L^{p'}(-1,0)}C_{q}t^{\frac{q-1}{2q}-\inv{2}}(1+\len h\pn_{L^p(-1,0;L^r(0,1))})\len u\pn_{L^2(0,1)},\notag
\end{align}
where $L_{g_1},L_{g_2}$ are, respectively, the Lipschitz constant of $g_1$ and $g_2$.
Now, by the $\gamma$-Fubini isomorphism (cf.  Proposition 2.6 in \cite{vanNeervenVeraarWeis}) we have $$J\lee\gamma\lee H,L^r(0,1)\p\p = L^r(0,1;H^*),$$
where the isomorphism  $J$ is given by $[u,(JR)(\xi)]_{H}=(Rh)(\xi)$ for every $R\in\gamma\lee H,L^r(0,1)\p$ for almost all $\xi\in[0,1]$ and all $u\in H$. Hence and by \eqref{ex_psi_1} and \eqref{ex_psi_2} for all $[x,h]\in \Ep{p}$ we have:
\begin{align}\label{exSphi_1}
\norm{S_r(t)\psi([x,h]')}_{\gamma(H,L^r(0,1))}\eqsim\left(\int_0^1\sup_{\n u\n_{H}\leq 1}\left|\left( S_r(t)\psi([x,h]')u\right)(\xi)\right|^rd\xi \right)^\inv{r}\leq b(t) (1+\len [x,h]\pn_{\Ep{p}}),
\end{align}
for all $t>0$, where $b(t)=2^{\inv{p'}}(L_{g_1}\maxsym L_{g_2}\operatorname{ess\ sup}_{\xi\in[0,1]} \len k_2(\cdot,\xi)\pn_{L^{p'}(-1,0)})C_{q}t^{\frac{q-1}{2q}-\inv{2}}$. Finally, observe that $b\in L^{p\maxsym 2}_{loc}(0,\infty)$.\par
The family of operators $\{W_H(t):t\geq 0\}$ is a $H$-cylindrical Wiener process by Theorem 3.2.4 in \cite{kallianpur1995stochastic}. 
\end{proof}
Therefore,  by Proposition \ref{p:exRD} and Corollary \ref{c:existSDE} if $p\geq1$ and $r>2$ are such that $p\maxsym 2<\frac{4r}{2+r}$, then there exists a unique weak solution to the problem \eqref{ex}. Moreover, if $p\maxsym 2<\frac{4r}{2+r\epsilon}$ for some $\epsilon\in(0,1)$, then by Corollary \ref{c:continuitySDE} the weak solution to the problem \eqref{ex} has a version in $L^p(\Omega;C([0,t];L^r(0,1)))$ for all $t>0$.

\subsubsection{Stochastic reaction-diffusion equation with unbounded delay}

In the Hilbert space $E=H=L^2(0,1)$ let us consider the following extended version of \eqref{ex}:
\begin{align}\label{ex2}
\left\{\begin{array}{l}
dy(t,\xi)=\Delta y(t,\xi)dt+\Big[\int_{t-1}^{t}\varphi(s-t,\xi)y(s,\xi)ds+\sum_{i=1}^n(\varphi_iy(t+\theta_i,\cdot))(\xi)\\ \quad+f_1(y(t,\xi))+ \int_{t-1}^{t}k_1(s-t,\xi)f_2(y(s,\xi))ds\Big]dt\\ \quad+\left(g_1(y(t,\xi))+\int_{t-1}^ {t}k_2(s-t,\xi)g_2(y(s,\xi))ds\right)d\mathcal{W}(t,\xi),
\\
y(t,0)=0, \quad y(t,1)=0,\\
y(0,\xi)=x_0(\xi), x_0\in L^r(0,1), \quad  y(\theta,\xi)=f_0(\theta,\xi), f_0\in L^p(-1,0;L^r(0,1)).
\end{array}\right.  
\end{align}  
where $\varphi_1,\ldots,\varphi_n\in\Ca{L}(E)$ and $-1=\theta_1<\theta_2<\ldots<\theta_n=0$.. Let $\Phi$ denote the delay operator:
$$(\Phi h)(\xi) =\int_{-1}^{0}\varphi(\theta,\xi)h(\theta,\xi)d\theta+\sum_{i=1}^n(\varphi_ih(\theta_i,\cdot))(\xi),$$
for all $h\in C([-1,0];E)$. Then, $\Phi$ satisfies \ref{h2} with $$\eta:[-1,0]\to\Ca{L}(L^2(0,1)),\ \eta(\theta)x(\xi)=\lee\int_{-1}^\theta \varphi(s,\xi)ds+\sum_{i:\theta_i\leq \theta}\varphi_i\p x(\xi),\ x\in E.$$ 
Form \cite[section 11.2.1]{DaPratoZabczykErgodicity} it follows that 
$\psi$ satisfies \ref{h4} with $$b(t)=\left(\sqrt{2}L_{g1}+\sqrt{2}L_{g_2}\operatorname{ess\ sup}_{\xi\in[0,1]} \len k_2(\cdot,\xi)\pn_{L^{p'}(-1,0)}\right)\len S(t) \pn_{\Ca{L}_2(E)}.$$ Notice that for all $t>0$ and all $p<4$ we have
\begin{align*}
\int_0^t \len S(s) \pn^{p\maxsym 2}_{\Ca{L}_2(E)}ds&= \int_0^t \left( \sum_{n=1}^\infty e^{-2s\pi^2n^2}\right)^{\frac{p\maxsym 2}{2}}ds\leq \left( \sum_{n=1}^\infty \left(\int_0^t e^{-(p\maxsym 2)s\pi^2n^2}ds\right)^\frac{2}{p\maxsym 2}\right)^\frac{{p\maxsym 2}}{2}\\&\leq
\inv{(p\maxsym 2)\pi^2}\left(\sum_{n=1}^\infty \inv{n^{\frac{4}{p\maxsym 2}}}\right)^\frac{p\maxsym 2}{2}<\infty.
\end{align*}
 Therefore,  by Corollary \ref{c:existSDE} if $p<4$, then there exists a unique weak solution to the problem \eqref{ex2}. Moreover, if $p<\frac{4}{1+\epsilon}$ for all $\epsilon>0$, then there exists $\alpha\in(0,\inv{p\maxsym 2}\frac{\epsilon}{\epsilon+1})$ such that for all $t>0$
\begin{align*}
 \int_0^t &\len S(s) \pn^{p\maxsym 2}_{\Ca{L}_2(E)}s^{-\alpha (p\maxsym 2)}ds\leq \lee\int_0^t \left( \sum_{n=1}^\infty e^{-2s\pi^2n^2}\right)^{(1+\epsilon)\frac{p\maxsym 2}{2}}\p^\inv{1+\epsilon}\lee\int_0^t s^{-\frac{\epsilon+1}{\epsilon}\alpha( p\maxsym 2)}ds\p^{\frac{\epsilon}{\epsilon+1}}\\&\leq
\lee\inv{(p\maxsym 2)\pi^2}\p^{\inv{1+\epsilon}}\left(\sum_{n=1}^\infty \inv{n^{\frac{4}{(1+\epsilon)(p\maxsym 2)}}}\right)^\frac{p\maxsym 2}{2}\lee\frac{t^{1-\frac{\epsilon+1}{\epsilon}\alpha(p\maxsym 2)}}{1-\frac{\epsilon+1}{\epsilon}\alpha(p\maxsym 2)}\p^{\frac{\epsilon}{1+\epsilon}}<\infty,  
\end{align*} 
by H\"older's inequality and Minkowski's integral inequality.
Finally, if $p<\frac{4}{1+\epsilon}$ for some $\epsilon>0$, then by Corollary \ref{c:continuitySDE} the weak solution to the problem \eqref{ex2} has a version in $L^p(\Omega;C([0,t];L^r(0,1)))$ for all $t>0$.

\section*{Acknowledgements}
The author wish to thank Anna Chojnowska-Michalik for helpful comments.

\bibliographystyle{plain}
\bibliography{literatura}

\end{document}